\documentclass[leqno,a4paper,10pt]{amsart}
\usepackage{amsmath,amsthm,amssymb,amsfonts,enumerate,color}
\usepackage[colorlinks, citecolor=blue,pagebackref,hypertexnames=false]{hyperref}
\allowdisplaybreaks[4]
\numberwithin{equation}{section}
\newtheorem{thm}{Theorem}[section]
\newtheorem{lem}[thm]{Lemma}
\newtheorem{remark}[thm]{Remark}
\newtheorem{prop}[thm]{Proposition}

\newcommand{\real}{{\mathbb R}}

\newcommand{\ent}{{\mathbb Z}}
\newcommand{\com}{{\mathbb C}}
\newcommand{\norm}[1]{\left\Vert#1\right\Vert}
\newcommand{\abs}[1]{\left\vert#1\right\vert}

\renewcommand{\H}{{\mathcal H}}

\newcommand{\M}{{\mathcal M}}
\renewcommand{\P}{{\mathcal P}}
\renewcommand{\a}{\alpha}
\newcommand{\e}{\varepsilon}
\newcommand{\f}{\varphi}
\newcommand{\8}{\infty}

\newcommand{\red}[1]{\textcolor{red}{#1}}

\begin{document}

\title[Differential transforms for parabolic operators]
{Boundedness of differential transforms for fractional Poisson type operators generated by parabolic operators}

\thanks{{\it 2020 Mathematics Subject Classification:} Primary: 42B20, 42B25.}
\thanks{{\it Key words:} differential transforms, fractional Poisson operators, maximal operator.}
\thanks{Supported by the National Natural Science Foundation of China(Grant No. 11971431, 11401525), the Natural Science Foundation of Zhejiang Province (Grant No. LY18A010006) and the first Class Discipline of Zhejiang-A(Zhejiang Gongshang University-Statistics).}

\author{Chao Zhang}

 \address{School of Statistics and Mathematics \\
             Zhejiang Gongshang University \\
             Hangzhou 310018, People's Republic of China}
 \email{zaoyangzhangchao@163.com}

%
%

\date{}
\maketitle

\begin{abstract}
In this paper we analyze the convergence of the following type of series
\begin{equation*}  T_N^\a f(x,t)=\sum_{j=N_1}^{N_2} v_j(\P_{a_{j+1}}^\a f(x,t)-\P_{a_j}^\a f(x,t)),\quad (x,t)\in \mathbb R^{n+1}, \ N=(N_1, N_2)\in \mathbb Z^2,\ \alpha>0,
\end{equation*}
where  $\{\P_{\tau}^\a \}_{\tau>0}$ is the fractional Poisson-type operators generated by the parabolic operator  $L=\partial_t-\Delta$ with $\Delta$ being the classical Laplacian,  $\{v_j\}_{j\in \mathbb Z}$  a bounded real sequences  and $\{a_j\}_{j\in \mathbb Z}$  an increasing real sequence.

Our analysis will consist  {of} the boundedness, in  $L^p(\mathbb{R}^n)$ and in  $BMO(\mathbb{R}^n)$,  of the operators
$T^{\alpha}_N$ and its maximal operator $ T^*f(x)= \sup_{N\in \mathbb Z^2} \abs{T^{\alpha}_N f(x)}.$

It is also shown that the local size of the maximal differential transform operators  is the same with the order  of a singular integral for functions $f$ having local support. Moreover, if $\{v_j\}_{j\in \mathbb Z}\in \ell^p(\mathbb Z)$, we get an intermediate size between the local size of singular integrals and Hardy-Littlewood maximal operator.
 \end{abstract}

\bigskip
\section{Introduction}

In this paper, we would like to develop the differential transform theory of some operators related with the operators which  is generated by the  parabolic operator $L=\partial_t -\Delta$. In this theory, the central role is the heat diffusion semigroup generated by $L$ that we denote by   $$e^{-\tau L}= e^{-\tau (\partial_t-\Delta)}= e^{-\tau\partial_t}\circ  e^{\tau \Delta},\quad\text{with } \tau>0.$$
Clearly, this semigroup can be given by an integral with a concrete kernel. That is, for any function ~$\f$ in Schwartz class $\mathcal S(\real^{n+1})$,
 \begin{align}\label{Formu:HeatPointwise}
 e^{-\tau (\partial_t-\Delta)}\varphi(x,t)=e^{\tau \Delta}\varphi(x,t-\tau)
                                             =\int_{\real^n} W(y,\tau)\varphi(x-y,t-\tau)dy,~(x,t)\in \real^{n+1},
\end{align}
where $W$ is the Gauss-Weierstrass kernel
  $$W(y,\tau)=\frac{1}{(4\pi \tau)^{n/2}} e^{-\frac{|y|^2}{4\tau}}.$$

\medskip

Let us  describe a kind of fractional Poisson formula, which will play a central role in our study. For $0<\a<1$, the \textit{fractional Poisson formula} of $f$ is given by
\begin{align}\label{Formu:GenPoisson}
\P_\tau^\a f(x)&=\frac{\tau^{2\alpha}}{4^\alpha\Gamma(\alpha)}
\int_0^\infty e^{-\tau^2/(4s)}e^{-sL} f(x)\,\frac{ds}{s^{1+\alpha}}\\
\nonumber &=\frac{1}{\Gamma(\alpha)}
\int_0^\infty e^{-r} e^{-\frac{\tau^2}{4r} L} f(x)\,\frac{d r}{r^{1-\alpha}}, \quad x\in\real^n, \ \tau>0.
\end{align}
This means that the fractional Poisson formula can be obtained via the heat semigroup $\displaystyle \{e^{-sL}\}_{s>0}$. The formula above can be  {deduced from} the theory of the spectral analysis. When $L=-\Delta$, Carffarelli and Silvestre \cite{CaffarelliSil} studied the fractional Poisson formula to solve an extension problems. Stinga and Torrea \cite{StingaTorreaExten} defined  this kind of Poisson formula for Hermite operator $L=-\Delta+|x|^2$. In the case $\a=1/2$, $\P_t^{1/2}$ is the Bochner subordinated Poisson semigroup, see \cite{SteinTopic}.

We will mainly study some properties of the differential transforms  associated to the operators generated by the parabolic operator. Let $\{a_j\}_{j\in \ent}$ be an increasing sequence of positive real numbers, and  $\{v_j\}_{j\in \ent}$ be a bounded sequence of real numbers. With this two sequences, we consider the differential transform
\begin{equation*}
 T  f =\sum_{j\in \ent} v_j(T_{a_{j+1}} f -T_{a_j} f),
\end{equation*}for   operators $\{T_s\}_{s>0}$.
In  this paper, we will study the properties of the differential transforms related to the fractional Poisson type operators  $\{\P^\a_{\tau} \}_{\tau>0}$ defined as (\ref{Formu:GenPoisson}). Moreover, combining (\ref{Formu:HeatPointwise}) and (\ref{Formu:GenPoisson}),  we have the pointwise formula as follows:
\begin{align*}
 \P_\tau^\a f(x,t) =\frac{\tau^{2\alpha}}{4^\alpha\Gamma(\alpha)}  \int_0^\infty \int_{\real^n}
           \frac{e^{-(\tau^2+|y|^2)/(4 s)}}{(4\pi s)^{n/2}} f (x-y,t-s)dy \frac{ds}{s^{1+\a}},
\end{align*}

\begin{remark}\label{Rem:FourierPoisson}
The Fourier transform $(\P_\tau^\a f)^{\wedge}$ can be characterized by the Bessel function or Macdonald function,  which is defined for arbitrary $\nu$ and $z\in\com$, see \cite{Lebedev}.  But in order to avoid complicated notations, whenever we have to deal with Fourier transform  {of} this kind fractional functions, we would compute directly the integrals by using the complex analysis.
\end{remark}

\medskip

Jones and Rosenblatt \cite{JR} studied the behavior of the series of the differences of ergodic averages and the differentials of differentiation operators along lacunary sequences in the context of the~$L^p$ spaces. In  \cite{BLMMDT}, the authors solved these problems with a different approach, which relied heavily on the method of Calder\'on-Zygmund singular integrals (see \cite{RubioRuTo}). In \cite{ZMT}, the authors proved the boundedness of the above operators related with the one-sided fractional Poisson type operator sequence. And in \cite{ZT}, the authors proved the boundedness of the differential transforms related to the heat semigroups generated by the Laplacian and Schr\"odinger operators.   We will focus on this kind of problems with the fractional Poisson type operators generated by the parabolic operator $L=\partial_t -\Delta.$

Let $\{a_j\}_{j\in \ent}$ be an increasing sequence of positive real numbers, and $\{v_j\}_{j\in \ent}$ be a bounded sequence of real  numbers. We consider the series
\begin{equation*}\label{Formu:SquareFun}
 \sum_{j\in \ent} v_j(\P_{a_{j+1}}^\a f(x,t)-\P_{a_j}^\a f(x,t)).
\end{equation*}
For each $N=(N_1,N_2)\in \ent^2$ with $N_1<N_2,$ we define the sum
\begin{equation}\label{Formu:FinSquareFun}
 T_N^\a f(x,t)=\sum_{j=N_1}^{N_2} v_j(\P_{a_{j+1}}^\a f(x,t)-\P_{a_j}^\a f(x,t)).
\end{equation}
Then, we have the following formula
\begin{align*}
&T_N^{\a} f(x,t) =\sum_{j=N_1}^{N_2} v_j(\P_{a_{j+1}}^\a f(x,t)-\P_{a_j}^\a f(x,t)) \\
     &= \frac{1}{4^\alpha\Gamma(\alpha)}  \sum_{j=N_1}^{N_2}v_j \int_{\real^n}\int_0^\8 \frac{ a_{j+1}^{2\a}e^{-a_{j+1}^2/(4 s)}-a_j^{2\a} e^{-a_j^2/(4 s)}        }{s^{1+\a}} \frac{e^{-|y|^2/(4s)}}{(4\pi s)^{n/2}} f(x-y,t-s)~ds dy \\
     &= \frac{1}{4^\alpha\Gamma(\alpha)} \int_{\real^{n+1}_+} \sum_{j=N_1}^{N_2}v_j \frac{ a_{j+1}^{2\a}e^{-a_{j+1}^2/(4 s)}-a_j^{2\a} e^{-a_j^2/(4 s)}        }{s^{1+\a}} \frac{e^{-|y|^2/(4s)}}{(4\pi s)^{n/2}} f(x-y,t-s)~ds dy.
\end{align*}
We denote the kernel of $T_N^\alpha$ by
$$K_N^\a(y,s)=\frac{1}{4^\alpha\Gamma(\alpha)} \sum_{j=N_1}^{N_2}v_j \frac{ a_{j+1}^{2\a}e^{-a_{j+1}^2/(4 s)}-a_j^{2\a} e^{-a_j^2/(4 s)}}{s^{1+\a}} \frac{e^{-|y|^2/(4s)}}{(4\pi s)^{n/2}}.$$
Note that the above function $K_N^\a(y,s)$ is defined for  $(y,s)\in \real^{n+1}_+$. But we can extend  $K_N^\a$ to the whole space $\real^{n+1}$. Indeed, we  observe that for any $y\in \real^n$, $K_N^\a(y,s)\to 0$ as $s\to 0^+$. Thus we can define $K_N^\a(y,s)=0$, for $s\leq 0$ and $y\in \real^n$.

\medskip
 In order to get the almost everywhere convergence of the series (\ref{Formu:FinSquareFun}), it is natural to study the corresponding maximal operator
\begin{equation*}
 T^*f(x,t)=\sup_N \abs{T_N^\alpha f(x,t)}, \quad (x,t)\in\real^{n+1},
\end{equation*}
where the supremum is taken over all $N=(N_1,N_2)\in \ent^2$ with $N_1< N_2$.

 Some of our results will be valid when the sequence  $\{a_j\}_{j\in \mathbb Z}$ is lacunary.  It means that there exists a $\rho >1$ such that $\displaystyle \frac{a_{j+1}}{a_j} \ge \rho, \, j \in \mathbb{Z}$. In particular, we shall prove  the boundedness of the operators $T^*$  in the weighted spaces
$L^p(\mathbb R^{n+1}, \omega),$ where $\omega$ is  {a} usual Muckenhoupt weights on $\mathbb R^{n+1}$. We refer the reader to the book by J. Duoandikoetxea \cite[Chapter 7]{Duo} for definition and properties of the $A_p$ classes.
And, we have the following results.

\begin{thm}\label{Thm:LpBoundOsci}
For the operator $T^*$, we have the following statements.
\begin{enumerate}[(a)]
    \item For any $1<p<\infty$ and $\omega\in A_p(\real^{n+1})$,  there exists a constant $C$ depending  on $n, \rho, \omega, \alpha, p$ and $\norm{v}_{\ell^\infty(\mathbb Z)}$ such that
 $$\norm{T^*f}_{L^p(\mathbb R^{n+1}, \omega)}\leq C\norm{f}_{L^p(\mathbb R^{n+1}, \omega)},$$
 for all functions $f\in L^p(\real^{n+1}, \omega).$
    \item For any  $\omega\in A_1(\real^{n+1})$, there exists a constant $C$ depending  on $n, \rho, \omega, \alpha$ and $\norm{v}_{\ell^\infty(\mathbb Z)}$ such that
 $${\omega\left(\{(x,t)\in \real^{n+1}:\abs{T^*f(x,t)}>\lambda\}\right)} \le C\frac{1}{\lambda}\norm{f}_{L^1(\mathbb R^{n+1}, \omega)}, \quad \lambda>0,$$
for all functions $f\in L^1(\real^{n+1}, \omega).$
    \item Given $f\in L^\infty(\real^{n+1}),$ then either $T^* f(x,t) =\infty$ for all $(x,t)\in \mathbb R^{n+1}$, or $T^* f(x, t) < \infty$ for $a. e.$  $(x,t)\in \mathbb R^{n+1}$. And in this {latter} case, there exists a constant $C$ depending  on $n, \rho$, $\alpha$ and  $\norm{v}_{\ell^\infty(\mathbb Z)}$ such that
        \begin{equation*}\norm{T^*f}_{BMO(\mathbb R^{n+1})}\leq C\norm{f}_{L^\infty(\mathbb R^{n+1})}.
        \end{equation*}
\item Given $f\in BMO(\real^{n+1}),$ then either $T^* f(x, t) =\infty$ for all $(x,t)\in \mathbb R^{n+1}$, or $T^* f(x, t) < \infty$ for $a. e.$  $(x,t)\in \mathbb R^{n+1}$.  And in this {latter} case,  there exists a constant $C$ depending  on $n, \rho$, $\alpha$ and  $\norm{v}_{\ell^\infty(\mathbb Z)}$ such that
\begin{equation}\label{sharp}\norm{T^*f}_{ BMO(\mathbb R^{n+1})}\leq C\norm{f}_{BMO(\mathbb R^{n+1})}.
\end{equation}
\end{enumerate}
\end{thm}

Then, by the theorems above, we can get the {\it{a.e.}} convergence result as follows.
\begin{thm}\label{Thm:ae}
\begin{enumerate}[(a)]
    \item If $1<p<\infty$ and $\omega\in A_p(\real^{n+1})$, then $T_N^\alpha f$ converges {\it{a.e.}} and in $L^p(\mathbb R^{n+1}, \omega)$ norm for all $f\in L^p(\mathbb R^{n+1}, \omega)$ as $N=(N_1,N_2)$ tends to $(-\infty, +\infty).$
    \item If $p=1$ and $\omega\in A_1(\real^{n+1})$, then $T_N^\alpha f$ converges {\it{a.e.}} and in measure for all $f\in L^1(\mathbb R^{n+1}, \omega)$ as $N=(N_1,N_2)$ tends to $(-\infty, +\infty).$
\end{enumerate}
\end{thm}

\medskip

At last, we will give the $L^\infty$ results of the differential transform   associated to the fractional Poisson type operators.
By Theorem \ref{Thm:LpBoundOsci},   $T^*$ is  bounded from $L^\infty(\real^{n+1})$ to $BMO(\real^{n+1})$  when $T^*f(x,t)<\infty$ $a.e. (x, t)\in \real^{n+1}$. We will give an example to see that,  there exists $f\in L^\infty(\mathbb{R}^{n+1})$ such that $T^*(f)=\infty$ $a.e$. Moreover, we can give the characterization of the local growth of the operator $T^*(f)$ in $L^\infty(\mathbb R^{n+1})$. These results are presented in Theorems \ref{Thm:LinfinityI} and \ref{Thm:GrothLinfinity}.

 \begin{thm}\label{Thm:LinfinityI}
 There exist bounded sequence $\{v_j\}_{j\in \mathbb Z}$, $\rho$-lacunary sequence $\{a_j\}_{j\in \mathbb Z}$  and $f\in L^\infty(\mathbb R^{n+1})$ such that  $T^* f(x,t) =\infty$ for all $(x,t)\in \mathbb R^{n+1}$.
\end{thm}

And the  characterization of the local growth of the operator $T^*$ in $L^\infty(\real^{n+1})$ is as follows:
\begin{thm} \label{Thm:GrothLinfinity}
\begin{enumerate}[(a)]
\item  Let   $\{v_j\}_{j\in \mathbb Z}\in l^p(\mathbb Z)$ for some $1 \le p\le \infty.$ For every $f\in L^\infty(\mathbb{R}^{n+1})$ with support in the cylinder $\tilde B=B(0, 1)\times[-1, 1]\subset \real^{n+1}$, for any cylinder $\tilde B_r:=B(0,r)\times[-r,r]\subset \tilde B$ with $2r<1$, we have
    $$\frac{1}{|\tilde B_r|} \int_{\tilde B_r} \abs{T^* f (x, t)} dxdt\leq C\left(\log \frac{2}{r}\right)^{1/p'}\norm{v}_{l^p(\mathbb Z)}\|f\|_{L^\infty(\mathbb R^{n+1})}.$$
\item When $1< p<\infty$, for any $\varepsilon>0$, there exist a $\rho$-lacunary sequence  $\{a_j\}_{j\in \mathbb Z}$,  a sequence $\{v_j\}_{j\in \mathbb Z}\in \ell^p(\mathbb Z)$ and a function  $f\in L^\infty(\mathbb{R}^{n+1})$ with support in  $\tilde B,$ satisfying the following statement: for any cylinder $\tilde B_r\subset \tilde B$ with $2r<1$, we have
    $$\frac{1}{|\tilde B_r|} \int_{\tilde B_r} \abs{T^* f (x, t)} dxdt\geq C\left(\log \frac{2}{r}\right)^{1/(p-\varepsilon)'}\norm{v}_{l^p(\mathbb Z)}\|f\|_{L^\infty(\mathbb R^{n+1})}.$$
\item When $p=\infty,$ there exist  a $\rho$-lacunary sequence $\{a_j\}_{j\in\mathbb Z}$, a sequence $\{v_j\}_{j\in \mathbb Z}\in \ell^\infty(\mathbb Z)$ and $f\in L^\infty(\mathbb{R}^{n+1})$ with support in  $\tilde B,$ satisfying the following statement: for any cylinder $\tilde B_r\subset \tilde B$ with $2r<1$, we have
    $$\frac{1}{|B_r|} \int_{B_r} \abs{T^* f (x, t)} dxdt\geq C\left(\log \frac{2}{r}\right)\norm{v}_{\ell^\infty(\mathbb Z)}\|f\|_{L^\infty(\mathbb R^{n+1})}.$$
\end{enumerate}
In the statements above,  $\displaystyle p' = \frac{p}{p-1},$ and if $p=1$, $\displaystyle p'=\infty.$
\end{thm}

The organization of the paper is as follows. Section \ref{Sec:L2} is devoted to {proving} the uniform boundedness  of the differential transforms $T_N^\alpha$. In Section \ref{Sec:Lp}, we give the proof of our main results, Theorem \ref{Thm:LpBoundOsci}, by proving a Cotlar's type inequality first and then using the vector-valued Calder\'on-Zygmund theory.  Theorem \ref{Thm:ae} is proved in Section \ref{Sec:ae}. At last,  we give the proof of the most interesting result in this paper, i.e. Theorems  \ref{Thm:LinfinityI} and \ref{Thm:GrothLinfinity}, in the {last} two sections.

\vskip 0.3cm
Throughout this paper, the symbol $C$ in an inequality always denotes a constant which may depend on some indices, but never on the functions $f$ in consideration.

\medskip

\section{Uniform boundedness of $T_N^\alpha$} \label{Sec:L2}

 In this section, we shall  deduce the uniform $L^p$-boundedness of the differential transform $T_N^\a$. Firstly, we should prove the $L^2$-boundedness of the differential transform $T_N^\a$ by Fourier transform. This result will be presented in {Proposition} \ref{Thm:L2Estimate}. And the following lemma gives a useful estimate which will be used in the later proofs.

\begin{lem}[{\cite[Lemma 2.1]{ZMT}}]\label{Lem:ComplexIntegral}
 Let $0<\a<1$. Then for any complex number $z_0$ with $Re z_0 > 0$ and $\displaystyle |\arg z_0 |\leq {\pi}/{4}$, we have
 $$ \int_0^\8  e^{-z_0 u} e^{-\frac{z_0}{u} }\,\frac{du}{u^{\alpha}}= z_0^{1-\a}\int_{0}^{\8}  \frac{e^{-r}e^{- z_0^2/r}}{r^{2-\a}} dr.$$
 \end{lem}

Now we present the uniform $L^2$-boundedness of the operator $T^\a_N$ in the following proposition.

\begin{prop}\label{Thm:L2Estimate}
 There exists a constant $C>0$, depending on $n$, $\alpha$ and $\norm{v}_{\ell^\infty}$, such that
 $$\sup_N \|T_N^\a f \|_{L^2(\real^{n+1})}\leq C \|f \|_{L^2(\real^{n+1})}.$$
\end{prop}

\begin{proof}
Let $f\in L^2(\real^{n+1})$. Using  the Plancherel theorem, we have
\begin{align*}
 \norm{T_N^\a f }_{L^2(\real^{n+1})} & = \norm{\sum_{j=N_1}^{N_2} v_j(\P_{a_{j+1}}^\a f -\P_{a_j}^\a f)}_{L^2(\real^{n+1})} \leq  C\|v_j\|_{\ell^\infty}
   \norm{\sum_{j=-\8}^{\8} \int_{a_j}^{a_{j+1}} \abs{\partial_\tau \widehat{\P_{\tau}^\a f }} d\tau}_{L^2(\real^{n+1})}.
\end{align*}
By using the second identity in \eqref{Formu:GenPoisson}, we have
\begin{align*}
\partial_\tau  \widehat{\P_{\tau}^\a f }(\xi,\varrho) & = C \partial_\tau \int_0^\8 e^{-r}
      \widehat{ e^{-\frac{\tau^2}{4r} L} f}(\xi,\varrho)\,\frac{dr}{r^{1-\alpha}}\\
    & = C \partial_\tau \int_0^\8 e^{-r} e^{-\frac{\tau^2}{4r}(i\varrho+|\xi|^2)}\widehat{f} (\xi,\varrho)
       \,\frac{dr}{r^{1-\alpha}}\\
    & = C  \int_0^\8 e^{-r}\tau (i\varrho+|\xi|^2) e^{-\frac{\tau^2}{4r}(i\varrho+|\xi|^2)}\widehat{f} (\xi,\varrho)
       \,\frac{dr}{r^{2-\alpha}}.
\end{align*}
Note that the Fourier transform above is well defined, see Remark \ref{Rem:ParaPoisson}. Then we deduce that
\begin{align*}
 \norm{T_N^\a f }_{L^2(\real^{n+1})} & \leq  C
   \norm{  \widehat{f} (\xi,\varrho) \int_{0}^\8 \abs{\int_0^\8 e^{-r}\tau (i\varrho+|\xi|^2) e^{-\frac{\tau^2}{4r}(i\varrho+|\xi|^2)}\,\frac{dr}{r^{2-\alpha}}  } d\tau}_{L^2(\real^{n+1})}.
\end{align*}
Thus again by the Plancherel theorem, the remainder is devoted to prove the uniform boundedness of the multiplier
$$\big|\widehat{K_N^\a}(\xi,\varrho)\big|\le C \abs{\int_{0}^\8\abs{ \int_0^\8 e^{-r}\tau (i\varrho+|\xi|^2) e^{-\frac{\tau^2}{4r}(i\varrho+|\xi|^2)}\,\frac{dr}{r^{2-\alpha}}} d\tau}\leq C, \quad (\xi,\varrho)\in \real^{n+1}_+.$$
Taking $z_0=\tau\sqrt{i\varrho+|\xi|^2}$, we rewrite the above inequality as
$$\abs{\widehat{K_N^\a}(\xi,\varrho)} \le C\int_{0}^\8 \abs{\int_0^\8  e^{-r}z_0 e^{-\frac{z_0^2}{4r} }
       \,\frac{dr}{r^{2-\alpha}}  \sqrt{i\varrho+|\xi|^2} }d\tau, \quad (\xi,\varrho)\in \real^{n+1}_+.$$
 By  Lemma \ref{Lem:ComplexIntegral}, for any $(\xi,\varrho)\in \real_+^{n+1}$, we have
 $$ \int_{0}^\8 \abs{\int_0^\8  e^{-r}z_0 e^{-\frac{z_0^2}{4r} }\,\frac{dr}{r^{2-\alpha}} \sqrt{i\varrho+|\xi|^2}} d\tau
    = 2^{1-\alpha}\int_{0}^\8 \abs{z_0^\a \int_0^\8  e^{-\frac{z_0}{2u}}e^{-\frac{z_0}{2}u} \frac{du}{u^\a} \sqrt{i\varrho+|\xi|^2} }d\tau.$$
 Since $\abs{\arg z_0}\leq \frac{\pi}{4}$, we have $|e^{-z_0/(2u)}| \leq e^{-c|z_0|/u} $ and $|e^{-z_0 u/2}| \leq e^{-c|z_0| u}$, where $c={ \sqrt{2}/ {4}}$. Then
\begin{align*}
\abs{\int_{0}^\8 z_0^\a \int_0^\8  e^{-z_0/u}e^{-z_0 u} \frac{du}{u^\a} \sqrt{i\varrho+|\xi|^2} d\tau} & \leq
  \int_{0}^\8\big|\sqrt{i\varrho+|\xi|^2}\big|\, |z_0|^\a \int_0^\8  e^{-c|z_0|/u}e^{-c|z_0| u} \frac{du}{u^\a} d\tau\\
&\leq \int_{0}^\8 \big|\sqrt{i\varrho+|\xi|^2}\big|\,|z_0|^{2\a-1}\int_0^\8  e^{-c|z_0|^2/v}e^{-c v} \frac{d v}{v^\a} d\tau.
\end{align*}
 Recall that $z_0=\tau \sqrt{i\varrho+|\xi|^2}$. Therefore, if we use $m$ to denote the module of $\sqrt{i\varrho+|\xi|^2}$: $m=|\sqrt{i\varrho+|\xi|^2}|$, we have
\begin{align*}
 &\int_{0}^\8 \abs{\sqrt{i\varrho+|\xi|^2}}\,|z_0|^{2\a-1}\int_0^\8  e^{-c|z_0|^2/v}e^{-c v} \frac{d v}{v^\a} d\tau\\
  & = \int_{0}^\8 m^{2\a}\,\tau^{2\a-1}\int_0^\8  e^{-c(m\tau)^2/v}e^{-c v} \frac{d v}{v^\a} d\tau\\
   & = \int_{0}^\8\int_0^\8 (m\tau)^{2\a-1} e^{-c(m\tau)^2/v} d(m\tau) e^{-c v} \frac{d v}{v^\a}\\
  & = \int_{0}^\8\int_0^\8  \tau^{2\a-1} e^{-c\tau^2/v} d\tau e^{-c v} \frac{d v}{v^\a}
    \leq C \int_{0}^\8   e^{-c v} d v \leq C,
\end{align*}
where the constants $C$ appeared above all are independent of $N.$
Then the proof of the proposition is complete.
\end{proof}

\begin{remark}\label{Rem:ParaPoisson}
Notice that the following integral
$$\frac{y^{2s}}{4^s\Gamma(s)}\int_0^\infty e^{-y^2/(4\tau)}e^{-\tau(i\varrho+\lambda)}\,\frac{d\tau}{\tau^{1+s}},
\quad\varrho\in\real,~\lambda\ge0,~0<s<1.$$
is convergent, as the Cauchy integral theorem and analytic continuation of the formula with $\varrho=0$ showing.
Notice that these integrals are related to Macdonald's functions, see \cite{Lebedev}.
Then \eqref{Formu:GenPoisson} are well defined by using Fourier transform and Hermite expansions, respectively.
\end{remark}

\bigskip

In this paper, we shall  use  the vector-valued Calder\'on-Zygmund theory in spaces of homogeneous type. One of the obvious changes in the parabolic setting is the metric of the underlying spaces. In our case, the parabolic distance is given by
\begin{equation*}\label{Formu:ParaDistance}
 d((x,t), (y,s))= \max \{|x-y|, |t-s|^{1/2}\}, \quad \mbox{ for }(x,t),(y,s)\in\real^n\times\real^1,
\end{equation*}
where the $\abs{\cdot}$ denotes the Euclidean distance.
Then $ \real^{n+1}= \real^n\times \real^1$  with the topology generated by the distance $d$ and the compatible Borel measure $dxdt$ forms a  space of homogeneous type. On the space of homogeneous type $(\mathbb R^{n+1}, d, dxdt)$, we can define the Calder\'on-Zygmund operators. For full details, see \cite{RuTo}, also \cite{LST} for parabolic case. In the following,  we shall show that the kernel $K_N^\alpha$ of the operator $T_N^\alpha$ is a Calder\'on-Zygmund kernel.

\begin{prop}\label{Thm:KernelEst}
There exists a constant $C$ depending  on $n, \alpha$ and $\norm{v}_{\ell^\infty}$(not on $N$) such that, for any $(y,s)\neq (0,0),$
\begin{enumerate}[\indent i)]
  \item $\displaystyle \abs{K_N^\a(y,s)}\leq \frac{C}{(s^{1/2}+|y|)^{n+2}}$,
  \item $\displaystyle \abs{\nabla_y K_N^\a(y,s)} \leq \frac{C}{(s^{1/2}+|y|)^{n+3}}$,
  \item $\displaystyle  \abs{\partial_s K_N^\a(y,s)}\leq \frac{C}{(s^{1/2}+|y|)^{n+4}}$.
\end{enumerate}
\end{prop}

The proof of {Proposition}  \ref{Thm:KernelEst} involves  a lemma in the following.
\begin{lem}\label{Lem:KernelLpEst}
Define a function $\displaystyle S(x,t)= \frac{e^{-{|x|^2\over c t}}}{t^{l+m}}$, $x\in \mathbb{R}^n,t\in\mathbb{R}_+$. Then,  for any integer $l, m \ge 1$, there exists a constant $C>0$, such that
$$\abs{S(x,t)} \leq \frac{C}{(t^{1/2}+|x|)^{2(l+m)}}.$$
\end{lem}
\begin{proof}
In fact, it is well known that, for any positive integer $m$, we have
   $$e^{-\frac{|x|^2}{c}}\leq C_m(1+|x|)^{-m}.$$
Then
\begin{align*}
 S(x,t) & =\frac{e^{-\frac{|x|^2}{ct}}}{t^{l+m}}=\frac{e^{-\frac1c \big(\frac{|x|}{\sqrt{t}}\big)^2}}{t^{l+m}}
                 \leq  \frac{C}{t^{l+m}} \big(1+\frac{|x|}{\sqrt{t}}\big)^{-2(l+m)}\\
         &\leq \frac{C}{t^{l+m}} \frac{t^{l+m}}{(t^{1/2}+|x|)^{2(l+m)}}\leq \frac{C}{(t^{1/2}+|x|)^{2(l+m)}}.
\end{align*}
The lemma is proved.
\end{proof}

\medskip
\begin{proof}[ Proof of {Proposition} \ref{Thm:KernelEst}.]
{\it i)}~This is just the growth condition for the kernel in the parabolic case. We have
\begin{align*}
  |K_N^\a (y,s)| &\leq C\sum_{j=-\8}^{\8} \left|\frac{a_{j+1}^{2\a} e^{-a_{j+1}^2/(4s)}-a_j^{2\a} e^{-a_j^2/(4s)}}{s^{1+\a}} \frac{e^{-|y|^2/(4s)}}{s^{n/2}} \right|\\
  &= C\sum_{j=-\8}^{\8} \left| a_{j+1}^{2\a} e^{-a_{j+1}^2/(4s)}-a_j^{2\a} e^{-a_j^2/(4s)}\right| \frac{e^{-|y|^2/(4s)}}{s^{1+\a+n/2}}.
\end{align*}

Observe that
\begin{align*}
&\sum_{j=-\8}^{\8} \left| a_{j+1}^{2\alpha}e^{-a_{j+1}^2/(4s)}-a_j^{2\alpha}e^{-a_j^2/(4s)}\right|
        = \sum_{j=-\8}^{\8} \left|\int_{a_j}^{a_{j+1}}\partial_u \left(u^{2\a} e^{-u^2/(4s)}\right) du\right|\\
       & \leq \int_{0}^{\infty}\left|( 2\a u^{2\a-1}-\frac{u^{2\a+1}}{2s})e^{-u^2/(4s)} \right| du \leq C\int_{0}^{\infty}\left|(u^{2\a-1}+\frac{u^{2\a+1}}{2s})e^{-u^2/(4s)} \right| du\\
       &\leq C\sqrt{s}\Big(\int_{0}^{\infty} (\sqrt{s} )^{2\a-1}  \left(\frac{u}{\sqrt{s}} \right)^{2\a-1} e^{-\frac14 \left(u/\sqrt{s}\right)^2} d\frac{u}{\sqrt{s}}\\
       &\quad\,+ s^{\a-1/2}\int_{0}^{\infty}\left(\frac{u}{\sqrt{s}}\right)^{2\a+1}  e^{-\frac14\left(u/\sqrt{s} \right)^2} d\frac{u}{\sqrt{s}}\Big)\\
       &\leq C s^{\a}.
\end{align*}
Then by Lemma \ref{Lem:KernelLpEst}, we have
\begin{align*}
  |K_N^\a (y,s)| & \leq C \frac{e^{-|y|^2/(4s)}}{s^{n/2+1}}\leq \frac{C}{(s^{1/2}+|y|)^{n+2}}.
\end{align*}

{The proof of {\it ii)} and {\it iii)} are similar. We omit them   here.}
\end{proof}

\begin{remark}\label{Rem:vectorCZK}
If we consider an $\ell^\infty(\mathbb Z^2)$-valued operator $Q: f\mapsto \left\{T_Nf(x,t)\right\}_{N\in \mathbb Z^2}$ on the homogeneous space $(\real^{n+1}, d, dxdt)$, then $T^* f(x,t)=\norm{Qf(x,t)}_{\ell^\infty(\mathbb Z^2)}$. Since the constant $C$ in {Proposition} \ref{Thm:KernelEst} is not depend on $N=(N_1, N_2)$, we know that the kernel of the operator $Q$ is an $\ell^\infty(\mathbb Z^2)$-valued Calder\'on-Zygmund kernel.
\end{remark}

From  Propositions  \ref{Thm:L2Estimate}, \ref{Thm:KernelEst}, and the vector-valued Calder\'on-Zygmund theory, we can get the uniform estimate in $L^p(\real^{n+1}, \omega)$ ($1 < p< \infty, \omega\in A_p(\real^{n+1})$) of the operators $T_N^\alpha$.

\begin{prop}\label{Thm:PoissonLp}
For the operator $T_N^\a$ defined in (\ref{Formu:FinSquareFun}), we have the following statements.
\begin{enumerate}[(a)]
    \item For any $1<p<\infty$ and $\omega\in A_p(\real^{n+1})$,  there exists a constant $C$ depending  on $n, p, \omega, \alpha$ and $\norm{v}_{\ell^\infty(\mathbb Z)}$ such that
 $$\norm{T_N^\alpha f}_{L^p(\mathbb R^{n+1}, \omega)}\leq C\norm{f}_{L^p(\mathbb R^{n+1}, \omega)},$$
 for all functions $f\in L^p(\real^{n+1}, \omega).$
    \item For any   $\omega\in A_1(\real^{n+1})$, there exists a constant $C$ depending  on $n, \omega, \alpha$ and $\norm{v}_{\ell^\infty(\mathbb Z)}$ such that
 $$\omega\left({\{(x,t)\in \real^{n+1}:\abs{T_N^\alpha f(x,t)}>\lambda\}}\right) \le C\frac{1}{\lambda}\norm{f}_{L^1(\mathbb R^{n+1}, \omega)},\quad \lambda>0,$$
for all functions $f\in L^1(\real^{n+1}, \omega).$
    \item There exists a constant $C$ depending  on $n, \alpha$ and $\norm{v}_{\ell^\infty(\mathbb Z)}$ such that
$$\norm{T_N^\alpha f}_{BMO(\mathbb R^{n+1})}\leq C\norm{f}_{L^\infty(\mathbb R^{n+1})},$$
for all functions $f\in L^\infty(\real^{n+1}).$
\item There exists a constant $C$ depending  on $n, \alpha$ and $\norm{v}_{\ell^\infty(\mathbb Z)}$ such that
$$\norm{T_N^\alpha f}_{BMO(\mathbb R^{n+1})}\leq C\norm{f}_{BMO(\mathbb R^{n+1})},$$
for all functions $f\in BMO(\real^{n+1}).$
\end{enumerate}
The constants $C$ appeared above all are independent of $N.$
\end{prop}

As we said before, the proof of $(a)$ and $(b)$ in  the proposition above is  obtained by using Theorem 7.12 in \cite{Duo}. On the other hand the proof of $(c)$ and $(d)$ are standard in the vector-valued Calder\'on-Zygmund theory and it can be found in \cite{Xu}.

\section{Proof of Theorem \ref{Thm:LpBoundOsci} }\label{Sec:Lp}


In this section, we will give the proof of the boundedness of the maximal operator $T^*.$
The next proposition,  parallel to  Proposition 3.2 in \cite{BLMMDT}(see it also in \cite{ZMT, ZT}), shows that, without lost of generality, we may assume that
\begin{equation}\label{equ:lacunary}
1<\rho \leq {a_{j+1} \over a_j}\leq \rho^2, \quad j\in \mathbb Z.
\end{equation}
We omit its proof  here.
\begin{prop}\label{Prop:lacunary}
Given a $\rho$-lacunary sequence $\{a_j\}_{j\in \mathbb Z}$ and a multiplying sequence $\{v_j\}_{j\in \mathbb Z}\in \ell^\infty(\mathbb Z)$, we can define a $\rho$-lacunary sequence $\{\eta_j\}_{j\in \mathbb Z}$ and $\{\theta_j\}_{j\in \mathbb Z}\in \ell^\infty(\mathbb Z)$ verifying the following properties:
\begin{enumerate}[(i)]
\item $1<\rho \leq \eta_{j+1}/\eta_j\leq \rho^2,\quad  \norm{\{\theta_j\}}_{\ell^\infty(\mathbb Z)}=\norm{\{v_j\}}_{\ell^\infty(\mathbb Z)}$.
\item For all $N=(N_1, N_2),$ there exists $N'=(N_1', N_2')$ with $T_N^\alpha=\tilde{T}_{N'}^\alpha,$
where $\tilde{T}_{N'}^\alpha$ is the operator defined in \eqref{Formu:FinSquareFun} for the new sequences $\{\eta_j\}_{j\in \ent}$ and $\{\theta_j\}_{j\in \ent}.$
\end{enumerate}
\end{prop}

It follows from this proposition that it is enough to prove all the results of this article
in the case of a $\rho$-lacunary sequence satisfying \eqref{equ:lacunary}. For this reason, in the rest of the article
we assume that  $\{a_j\}_{j\in \mathbb Z}$  satisfies \eqref{equ:lacunary} without saying it explicitly.

In order to prove Theorem \ref{Thm:LpBoundOsci}, we need a Cotlar's type inequality to control the operator $T^*$ by some  Hardy-Littlewood maximal operators.

For any $M\in \mathbb Z^+,$ let\\
$$T_M^*f(x,t)=\sup_{-M\le N_1<N_2\le M}\abs{T_N^\alpha f(x,t)},\quad (x,t)\in \mathbb R^{n+1}.$$
And then we have
$$T^*f(x,t)=\sup_{M\in \mathbb Z^+} T_M^*f(x,t).$$
Then we can start proving a pointwise estimate for the operators $T_M^*f$ by the Hardy-Littlewood maximal operator defined as
$$\M_q f(x,t)=\sup_{B\ni x}\left(\frac{1}{\abs B}\int_B\abs{f(y,t)}^qdy\right)^{1/q},\quad (x, t)\in \real^{n+1},$$
and the maximal operator
$$\M^-_q f(x,t)=\sup_{\varepsilon >0}\left(\frac{1}{\varepsilon}\int_{-\varepsilon}^0 \abs{f(x,t+s)}^qds\right)^{1\over q},$$
for $ 1\le q<\infty$. And we denote $\M=\M_1$ and $\M^-=\M^-_1$, for simple.

\begin{thm}\label{Thm:Maximalcontrol}
For each $q\in (1, \infty),$ there exists a constant $C$ depending only on $\alpha, \rho,$ $\norm{v}_{\ell^\infty}$ and $n$, such that  for every $(x,t)\in\mathbb R^{n+1}$ and every $M\in \mathbb Z^+$,
\begin{equation*}
T_M^*f(x,t)\le C\left\{\left(\M^-\circ \M\right) (T_{(-M, M)}^\alpha f)(x,t)+\left(\M_q^-\circ\M_q\right) f(x,t)\right\}.
\end{equation*}
\end{thm}

For the proof of this theorem we shall need the following lemma.

\begin{lem}\label{lem:cotlar}  Let  $\{a_j\}_{j\in \mathbb Z}$ be a $\rho$-lacunary sequence and $\{v_j\}_{j\in \mathbb Z} \in \ell^\infty(\mathbb Z)$. Then
\begin{itemize}
\item[(i)] $\displaystyle \abs{\sum_{j=m}^{M}v_j \left(\frac{ a_{j+1}^{2\a}e^{-a_{j+1}^2/(4 s)}-a_j^{2\a} e^{-a_j^2/(4 s)}}{s^{1+\a}} \right) } \le { {C_{\rho, v, \alpha}} \over a_m^{2}}, $
\

\item[(ii)] if $k\ge m$ and $t,s\in \real$  with $|t-s|\ge c a_k^{2}(c>0)$,  \begin{align*}
\abs{\sum_{j=-M}^{m-1}v_j \frac{ a_{j+1}^{2\a}e^{-{a_{j+1}^2\over 4 (t-s)}}-a_j^{2\a} e^{-{a_j^2\over 4 (t-s)}}}{(t-s)^{1+\a}}} \, \le C_{\rho, v, \alpha}\frac1{a_k^{2}}\rho^{-2\alpha(k-m+1)}.
\end{align*}

\end{itemize}
\end{lem}

\begin{proof}
For $(i)$, since $ \rho\le \frac{a_{j+1}}{a_j}\le \rho^2,$ we have
\begin{multline*}
\abs{\sum_{j=m}^{M}v_j \left(\frac{ a_{j+1}^{2\a}e^{-a_{j+1}^2/(4 s)}-a_j^{2\a} e^{-a_j^2/(4 s)}}{s^{1+\a}} \right) }
\le  \norm{v}_{\ell^\infty(\mathbb Z)} \sum_{j=m}^{M} \left({1\over a_{j+1}^2}+{1\over a_j^2}\right)
\\  \le C_{v,\alpha}{1\over a_m^{2}}  \sum_{j=m}^{M}  {\frac{\rho^2+1}{\rho^2}}\cdot {a_m^{2}\over a_j^2}= C_{\rho, v, \alpha} {1\over a_m^{2}} \sum_{j=m}^{M}  \abs{a_m\over a_j}^{2}   \le C_{\rho, v, \alpha}  {1\over a_m^{2}} {\rho^{2}\over \rho^{2}-1 }  \le C_{\rho, v, \alpha} {1\over a_m^{2}}.
\end{multline*}

Now we shall prove $(ii)$. By the mean value theorem, there exist $\xi_j$ with $a_j\le \xi_j\le a_{j+1}$ such that
{\begin{align*}
& \abs{\sum_{j=-M}^{m-1}v_j \frac{ a_{j+1}^{2\a}e^{-{a_{j+1}^2\over 4 (t-s)}}-a_j^{2\a} e^{-{a_j^2\over 4 (t-s)}}}{(t-s)^{1+\a}}}=\abs{\sum_{j=-M}^{m-1}v_j \frac{ (a_{j+1}-a_j)2\xi_j^{2\alpha-1}e^{-{\xi_j^2\over 4 (t-s)}} \left(\alpha-\frac{\xi_j^2}{4(t-s)}  \right)}{(t-s)^{1+\a}}}\\
&\le  \norm{v}_{\ell^\infty(\mathbb Z)} \left(\abs{\sum_{j=-M}^{m-1} \frac{(a_{j+1}-a_j)2\alpha\xi_j^{2\alpha-1}  e^{-{\xi_j^2\over 4 (t-s)}}}{(t-s)^{1+\a}}}+ \abs{\sum_{j=-M}^{m-1} \frac{(a_{j+1}-a_j)2\xi_j^{2\alpha+1}  e^{-{\xi_j^2\over 4 (t-s)}}}{(t-s)^{2+\a}}}\right)\\
&\le C_{\rho, v, \alpha}\left( \sum_{j=-M}^{m-1} {\frac{ a_j^{2\alpha} }{\abs{t-s}^{1+\alpha}} } + \sum_{j=-M}^{m-1} {\frac{ a_j^{2\alpha+2} }{\abs{t-s}^{2+\alpha}} }  \right)
\le C_{\rho, v, \alpha} {1\over a_k^{2}}\left(\sum_{j=-M}^{m-1} \abs{\frac{ a_j }{a_k }}^{2\alpha}+ \sum_{j=-M}^{m-1} \abs{\frac{ a_j }{a_k }}^{2\alpha+1}   \right)\\
&\le C_{\rho, v, \alpha}  {1\over a_k^{2}}\rho^{-2\alpha(k-m+1)},
\end{align*}}
where we have used that $k \ge m.$
\end{proof}

With Lemma \ref{lem:cotlar}, we can give the proof of Theorem \ref{Thm:Maximalcontrol}.

\begin{proof}[Proof of Theorem \ref{Thm:Maximalcontrol}]
Since the operators $T_N^\alpha$ are given by convolutions, they are invariant under translations, and therefore it is enough to prove the theorem for $(x,t)=(0,0).$ Observe that, for $N=(N_1, N_2),$
$$T_N^\alpha f(x,t)=T_{(N_1, M)}^\alpha f(x,t)-T_{(N_2+1, M)}^\alpha f(x,t),$$
with $-M\le N_1<N_2\le M.$
Then, it suffices to estimate $\abs{T_{(m, M)}^\alpha f(0,0)}$ for $\abs{m}\le M$ with constants independent of $m$ and $M.$
Let $B_m=B(0, a_m)$ denote the  ball with center $0$ and radius $a_m$ in $\real^{n}$, and let $\tilde{B}_m=B_m\times [-a_m^2, 0]\subset \real^{n+1}$. Then, we decompose
\begin{align*}
f&=f\chi_{\tilde B_m}+f\chi_{{\tilde B_m}^c}
=f\chi_{\tilde B_m} +f\chi_{B^c_m\times (-\infty, -a_m^2)}+f\chi_{B^c_m\times (0, +\infty)}=:f_1+f_2+f_3.
\end{align*}
We should note that $T_{(m, M)}^\alpha f_3(0,0)=0$ by the definition of $T_{(m, M)}^\alpha.$ Then, we have
\begin{align*}
\abs{T_{(m,M)}^\alpha f(0,0)}&\le \abs{T_{(m,M)}^\alpha f_1(0,0)}+\abs{T_{(m,M)}^\alpha f_2(0,0)}=: I+II.
\end{align*}
For $I$, by Lemma \ref{lem:cotlar} we have
\begin{align*}
I&=C_{n,\alpha} \abs{\int_{\real^{n+1}} \sum_{j=m}^{M}v_j \frac{ a_{j+1}^{2\a}e^{-a_{j+1}^2/(4 s)}-a_j^{2\a} e^{-a_j^2/(4 s)}}{s^{1+\a}} \frac{e^{-|y|^2/(4s)}}{(4\pi s)^{n/2}}  f_1(-y,-s)dyds}\\
&\le C_{n, \alpha} \norm{v}_{\ell^\infty(\mathbb Z)}\int_{-a_m^2}^0\int_{B_m} \sum_{j=m}^{M} \abs{\frac{ a_{j+1}^{2\a}e^{-a_{j+1}^2/(4 s)}-a_j^{2\a} e^{-a_j^2/(4 s)}}{s^{1+\a}}} {\frac{e^{-|y|^2/(4s)}}{(4\pi s)^{n/2}}}  \abs{f(-y,-s)}dyds\\
&\le C_{n,\alpha,v,\rho} \frac{1}{a_m^2}\int_{-a_m^2}^0\int_{B_m} {\frac{e^{-|y|^2/(4s)}}{(4\pi s)^{n/2}}}  \abs{f(-y,-s)}dyds\\
&\le C_{n,\alpha,v,\rho} \frac{1}{a_m^2}\int_{-a_m^2}^{0} \M_q f(0, -s))  ds\\
&\le C_{n,\alpha,v,\rho}  \left(\M_q^-\circ\M_q\right) f(0,0).
\end{align*}
For part $II$,
\begin{align*}\label{equ:II}
II&= \frac{1}{\abs{\tilde B_{m-1}}} \int_{\tilde B_{m-1}}\abs{T_{(m,M)}^\alpha f_2(0,0)}dxdt\\
&\le \frac{1}{\abs{\tilde B_{m-1}}} \int_{\tilde B_{m-1}}\abs{T_{(-M,M)}^\alpha f(x,t)}dxdt\\
&\quad +\frac{1}{\abs{\tilde B_{m-1}}} \int_{\tilde B_{m-1}}\abs{T_{(-M,M)}^\alpha f_1(x,t)}dxdt\\
&\quad +\frac{1}{\abs{\tilde B_{m-1}}} \int_{\tilde B_{m-1}}\abs{T_{(m,M)}^\alpha f_2(x,t)-T_{(m,M)}^\alpha f_2(0,0)}dxdt\\
&\quad +\frac{1}{\abs{\tilde B_{m-1}}} \int_{\tilde B_{m-1}}\abs{T_{(-M,m-1)}^\alpha f_2(x,t)}dxdt\\
&=:A_1+A_2+A_3+A_4.
\end{align*}
(If $m=-M$, we understand that $A_4=0$.) It is clear that
\begin{equation*}
A_1=\frac{1}{\abs{\tilde B_{m-1}}}\int_{-a_{m-1}^2}^0 \int_{B_{m-1}}\abs{T_{(-M,M)}^\alpha f(x,t)}dxdt\le C \left(\M^-\circ \M\right) (T_{(-M,M)}^\alpha f)(0,0).
\end{equation*}
For $A_2,$ by H\"older's inequality and the uniform $L^q$-boundedness of $T_N^\alpha$, we get
\begin{align*}
A_2&\le \left( \frac{1}{\abs{\tilde B_{m-1}}} \int_{\tilde B_{m-1}}\abs{T_{(-M,M)}^\alpha f_1(x,t)}^qdxdt\right)^{1/q}\\
&\le C \left( \frac{1}{\abs{\tilde B_{m-1}}} \int_{\real^{n+1}}\abs{ f_1(x,t)}^qdxdt\right)^{1/q}\\
&\le C\left(\M_q^-\circ \M_q \right) f(0,0).
\end{align*}
For $A_3$, with $(x,t)\in \tilde B_{m-1}$ we have
\begin{align*}
&\abs{T_{(m,M)}^\alpha f_2(x,t)-T_{(m,M)}^\alpha f_2(0,0)}\\
&=\abs{\int_{-a_{m}^2}^0\int_{B^c_{m}}K_{(m, M)}^\alpha(x-y,t-s)f(y,s)dyds -\int_{-a_{m}^2}^0\int_{B^c_{m}}K_{(m, M)}^\alpha(-y,-s)f(y,s)dyds}\\
&\le \int_{-a_{m}^2}^0\int_{B^c_{m}}\abs{K_{(m, M)}^\alpha(x-y,t-s)-K_{(m, M)}^\alpha(-y,-s)}\abs{f(y,s)}dyds\\
&=\sum_{k=m}^{+\infty}\int_{-a_{k+1}^2}^{-a_{k}^2}\int_{B_{k+1}\setminus B_k}\abs{K_{(m, M)}^\alpha(x-y,t-s)-K_{(m, M)}^\alpha(-y,-s)}\abs{f(y,s)}dyds.
\end{align*}
By the mean value theorem, we know that  there exists  $\theta$ with $0<\theta<1$  such that
\begin{align}\label{equ:mean}
&\sum_{k=m}^{+\infty}\abs{K_{(m, M)}^\alpha(x-y,t-s)-K_{(m, M)}^\alpha(-y,-s)}\nonumber \\
&\le \sum_{k=m}^{+\infty}\big(\abs{\bigtriangledown_{z} K_{(m, M)}^\alpha(z, \theta t-s)\big|_{z=\theta x-y}}|x|+\abs{\partial_{r} K_{(m, M)}^\alpha(\theta x-y, r)\big|_{r=\theta t-s}}|t|\big)\nonumber \\
&{\le C\sum_{k=m}^{+\infty} \left({|x|\over \left(\abs{\theta x-y}+\abs{\theta t-s}^{1/2}\right)^{n+3}}+{|t|\over \left(\abs{\theta x-y}+\abs{\theta t-s}^{1/2}\right)^{n+4}}\right)}\\
&\le C\sum_{k=m}^{+\infty} {a_{m-1}\over a_k^{n+3}},\nonumber
\end{align}
where we have used that in each summand $(x,t)\in \tilde B_{m-1}$ and $(y,s)\in \tilde B_{k+1}\backslash \tilde B_{k}$.
Hence
\begin{align*}
&\abs{T_{(m,M)}^\alpha f_2(x,t)-T_{(m,M)}^\alpha f_2(0,0)}
\le C\sum_{k=m}^{+\infty}\int_{-a_{k+1}^2}^{-a_{k}^2}\int_{B_{k+1}\setminus B_k} {a_{m-1}\over a_k^{n+3}}\abs{f(y,s)}dyds\\
&\le C\sum_{k=m}^{+\infty}{a_{m-1}\over a_k}\cdot {1\over a_k^{n+2}}\int_{-a_{k+1}^2}^{-a_{k}^2}\int_{B_{k+1}}\abs{f(y,s)}dyds\\
&\le C\sum_{k=m}^{+\infty}{a_{m-1}\over a_k}\cdot {\rho^{2n+2}\over a_{k+1}^{2}}\int_{-a_{k+1}^2}^{-a_{k}^2}\left({1\over a_{k+1}^n}\int_{B_{k+1}}\abs{f(y,s)}dy\right)ds\\
&\le C\sum_{k=m}^{+\infty}{a_{m-1}\over a_k}\cdot {1\over a_{k+1}^{2}}\int_{-a_{k+1}^2}^{-a_{k}^2}\M f(0,s)ds\\
&\le C \left(\M^-\circ\M\right) f(0,0)\sum_{k=m}^{+\infty}{a_{m-1}\over a_k}\\
&\le C\left(\M^-\circ\M\right) f(0,0).
\end{align*}
Therefore,
$$A_3\le C\left(\M^-\circ\M\right) f(0,0).$$
For the last term $A_4,$  we have
\begin{align*}
A_4
&=\frac{1}{\abs{\tilde B_{m-1}}} \int_{\tilde B_{m-1}}\abs{\int_{\tilde B_m^c} K_{(-M, m-1)}^\alpha (x-y, t-s)f(y,s)dyds}dxdt\\
&\le \frac{1}{\abs{\tilde B_{m-1}}} \int_{\tilde B_{m-1}}\int_{\tilde B_m^c} \abs{K_{(-M, m-1)}^\alpha (x-y, t-s)}\abs{f(y,s)}dyds~dxdt.
\end{align*}
Since  $(y,s)\in \tilde B_m^c,$ $(x,t)\in \tilde B_{m-1}$, and $\{a_j\}_{j\in \mathbb Z}$ is a $\rho$-lacunary sequence, we have $\abs{(x-y, t-s)}\sim \abs{(y, s)}.$ Then, by Lemma \ref{lem:cotlar} we get
\begin{align*}
&\int_{\tilde B_m^c} \abs{K_{(-M, m-1)}^\alpha (x-y, t-s)}\abs{f(y,s)}dyds\\
&\le C_{n,\alpha} \sum_{k=m}^{+\infty} \int_{\tilde B_{k+1}\setminus \tilde B_{k}} \abs{\sum_{j=-M}^{m-1}v_j \frac{ a_{j+1}^{2\a}e^{-{a_{j+1}^2\over 4 (t-s)}}-a_j^{2\a} e^{-{a_j^2\over 4 (t-s)}}}{(t-s)^{1+\a}}} \frac{e^{-{|x-y|^2\over 4|t-s|}}}{(4\pi |t-s|)^{n/2}}\abs{f(y,s)}dyds\\
&\le C_{n, \alpha, v, \rho} \sum_{k=m}^{+\infty} {\rho^{-2\alpha(k-m+1)}\over a_{k+1}^2} \int_{-a_{k+1}^2}^{-a_k^2}\int_{ B_{k+1}}  \frac{e^{-{|x-y|^2\over 4|s|}}}{(4\pi |s|)^{n/2}}\abs{f(y,s)}dyds\\
&\le C_{n, \alpha, v, \rho} \sum_{k=m}^{+\infty}\rho^{-2\alpha(k-m+1)}\cdot {1\over a_{k+1}^2} \int_{-a_{k+1}^2}^{0} \M f(0,s )ds\\
&\le C_{n, \alpha, v, \rho}\left(\M^-\circ \M\right) f(0,0) \sum_{k=m}^{+\infty}\rho^{-2\alpha(k-m+1)}\\
&\le C_{n, \alpha, v, \rho}\left(\M^-\circ \M\right) f(0,0).
\end{align*}
Hence, $$A_4\le C \left(\M^-\circ \M\right) f(0,0).$$
Combining the estimates above for $A_1, A_2, A_3$ and $A_4$, we get
$$II\le C \left(\M_q^-\circ \M_q\right) f(0,0).$$
And hence
$$ \abs{T_{(m,M)}^\alpha f(0,0)}\le C\left( \M^-\circ \M\right) (T_{(-M,M)}^\alpha f)(0,0)+C\left(\M_q^-\circ \M_q\right) f(0,0). $$
As the constants $C$ appeared above all only depend on $n, \alpha, \rho$ and $\norm{v}_{\ell^\infty(\mathbb Z)}$, we complete the proof of the theorem.
\end{proof}
Now,  we can start  the proof of Theorem \ref{Thm:LpBoundOsci}.
\begin{proof}[Proof of Theorem \ref{Thm:LpBoundOsci}]
$(a).$ For each $\omega\in A_p(\real^{n+1}),$ we can choose $1<q<p$ such that $\omega\in A_{p/q}(\real^{n+1}).$ Then,  it is well known that the maximal operators $\M_-\circ  \M,$ $\M_q^-\circ \M_q$ are bounded on $L^p(\real^{n+1}, \omega)$. Then, by Theorem \ref{Thm:Maximalcontrol} and the  uniform $L^p$-boundedness of $T_N^\alpha$ in {Proposition} \ref{Thm:PoissonLp}, we have
\begin{align*}
\norm{T_M^*f}_{L^p(\real^{n+1}, \omega)}&\le C\left(\norm{\left( \M^-\circ \M\right) (T_{(-M, M)}^\alpha f)}_{L^p(\real^{n+1}, \omega)}+\norm{\left(\M_q^-\circ \M_q\right) f}_{L^p(\real^{n+1}, \omega)}\right)\\
&\le C\left(\norm{T_{(-M, M)}^\alpha f}_{L^p(\real^{n+1}, \omega)}+\norm{f}_{L^p(\real^{n+1}, \omega)}\right)\le C\norm{f}_{L^p(\real^{n+1}, \omega)}.
\end{align*}
We should note that the constants $C$ appeared above do not depend on $M$. { And the     operator $T_M^*$ is monotonic  with respect to $M$.} Consequently, letting $M$ increase to infinity, we get the proof of the $L^p$-boundedness of $T^*$. This completes the proof of part $(a)$ of  Theorem \ref{Thm:LpBoundOsci}.

In order to prove $(b)$, we consider the $\ell^\infty(\mathbb Z^2)$-valued operator
$\mathcal{T}f(x, t) = \{  T_N^\alpha f(x, t) \}_{N\in \mathbb Z^2}$. Since $\|\mathcal{T}f(x,t) \|_{\ell^\infty(\mathbb Z^2)}= T^*f(x,t)$,   by using $(a)$ we know that the operator $\mathcal{T}$ is bounded from $L^p(\mathbb{R}^{n+1}) $ into $L^p_{\ell^\infty(\mathbb Z^2)}(\mathbb{R}^{n+1}, \omega) $, for every $1<p<\infty$ and $\omega\in A_p(\real^{n+1})$. The kernel of the operator $\mathcal{T}$ is given by $\mathcal{K}^\alpha(t) = \{ K^\alpha_N(t)\} _{N\in \mathbb Z^2}$. By {Proposition} \ref{Thm:KernelEst} and the vector valued version of Theorem 7.12 in \cite{Duo}, we get that the operator $\mathcal{T}$ is bounded from $L^1(\mathbb{R}^{n+1}, \omega)$ into weak- $L^1_{\ell^\infty(\mathbb Z^2)}(\mathbb{R}^{n+1}, \omega)$ for $\omega\in A_1(\real^{n+1})$. Hence, as $\|\mathcal{T}f(x, t) \|_{\ell^\infty(\mathbb Z^2)}= T^*f(x, t)$, we get the proof of  $(b)$.

For $(c)$  and $(d)$, since $L^\infty(\real^n)\subset BMO(\real^{n}),$ we only need to prove $(d).$ Let $(x_0,t_0)\in \real^{n+1}$ be one point such that  $T^*f(x_0,t_0)<\infty.$ Given $(x,t)\neq (x_0,t_0),$ denote $d_0=d((x,t), (x_0,t_0)).$
Set $B=B(x_0, 4d_0)$ and $\tilde B=B\times [t_0-4d_0^2, t_0+4d_0^2]$. And we decompose $f$ to be
\begin{equation*}
f=(f-f_{\tilde B})\chi_{\tilde B}+(f-f_{\tilde B})\chi_{{\tilde B}^c}+f_{\tilde B}=:f_1+f_2+f_3.
\end{equation*}
Note that $T^*$ is $L^p$-bounded for any $1<p<\infty.$ Then $T^*f_1(x,t)<\infty$, because $f_1\in L^p(\mathbb R^{n+1})$ for any $1<p<\infty.$ And  $T^* f_3=0$, since $\P_{a_j}^\alpha f_3=f_3$ for any $j\in \mathbb Z.$
On the other hand, by the same argument in \eqref{equ:mean}, we have
\begin{align*}
&\Big|T_N^\alpha f_2(x,t)-T_N^\alpha f_2(x_0,t_0)\Big|\\&=\Big|\int_{\real^{n+1}} K_N^\alpha (x-y, t-s)f_2(y,s)dyds-\int_{\real^{n+1}} K_N^\alpha (x_0- y, t_0-s)f_2(y,s)dyds \Big|\\
&=\Big| \int_{\tilde B^c}\left(K_N^\alpha(x-y, t-s) - K_N^\alpha (x_0- y, t_0-s)\right)f_2(y,s)dyds\Big|\\
&\le C \int_{\tilde B^c} \left( \frac{d_0}{\abs{(y,s)-(x_0,t_0)}^{n+3}}+ \frac{d_0^2}{\abs{(y,s)-(x_0,t_0)}^{n+4}}\right)\abs{f(y,s)-f_{\tilde B}}dyds  \\
&\le C \sum_{k=1}^{+\infty}{ d_0} \int_{2^{k} \tilde B\setminus 2^{k-1}\tilde B}  {\abs{f(y,s)-f_{\tilde B}}\over |(y,s)-(x_0,t_0)|^{n+3}}dyds\\
&\le C \sum_{k=1}^{+\infty}{ d_0\over (2^{k+1}d_0)^{n+3}} \int_{2^k\tilde B}  {\abs{f(y,s)-f_{\tilde B}}}dy\\
&\le C \sum_{k=1}^{+\infty}2^{-(k+1)}{ 1\over \abs{2^{k}\tilde B}} \int_{2^{k}\tilde B} \left(  {\abs{f(y,s)-f_{2^k\tilde B}}}+\sum_{l=1}^{k}\abs{f_{2^l\tilde B}-f_{2^{l-1}\tilde B}}\right)dy\\
&\le C\sum_{k=1}^{+\infty}2^{-(k+1)}{ 1\over \abs{2^{k}\tilde B}} \int_{2^{k}\tilde B} \left(  {\abs{f(y,s)-f_{2^k\tilde B}}}+2k\norm{f}_{BMO(\real^{n+1})}\right)dy\\
&\le C\sum_{k=1}^{+\infty}2^{-(k+1)}{(1+2k)\norm{f}_{BMO(\real^{n+1})}}\\
&\le  C\norm{f}_{BMO(\real^{n+1})},
 \end{align*}
where $2^k\tilde B=B(x_0, 2^{k}\cdot 4d_0)\times[t_0-2^k\cdot 4d_0^2, t_0+2^k\cdot 4d_0^2]$ for any $k\in \mathbb N.$
 Hence
  \begin{align*}
\norm{T_N^\alpha f_2(x,t)-T_N^\alpha f_2(x_0,t_0)}_{\ell^\infty(\mathbb Z^2)} \le C\norm{f}_{BMO(\mathbb R^{n+1})}
\end{align*}
and therefore $ T^*f(x,t) = \norm{T_N^\alpha f(x,t)}_{\ell^\infty(\mathbb Z^2)}  \le C < \infty.$

Now, we shall prove  the  estimate (\ref{sharp}) for functions such that $T^*f(x,t) < \infty \, \, a.e.$
For any $r>0$ and $(x_0,t_0)$ such that $T^*f(x_0,t_0) < \infty$,  let  $B=B(x_0, r)$, $\tilde B=B\times[t_0-r^2, t_0+r^2]$ and $\displaystyle f_{\tilde B}={1\over |\tilde B|}\int_{\tilde B} f(x,t)dxdt.$  Let
$$f=(f-f_{\tilde B})\chi_{2{\tilde B}}+(f-f_{\tilde B})\chi_{(2{\tilde B})^c}+f_{\tilde B}=:f_1+f_2+f_3.$$
  We have $T^*f_3(x,t)=0.$
And,
\begin{align*}
&{1\over |{\tilde B}|}\int_{{\tilde B}}\abs{T^*f(x,t)-(T^*f)_{\tilde B}}dxdt={1\over |{\tilde B}|}\int_{{\tilde B}}\abs{{1\over |{\tilde B}|}\int_{{\tilde B}}\left(T^*f(x,t)-T^*f(y,s)\right)dyds}dxdt\\
&\le {1\over |{\tilde B}|^2}\int_{{\tilde B}}\int_{{\tilde B}}\abs{T^*f(x,t)-T^*f(y,s)}dyds~dxdt\\
&={1\over |{\tilde B}|^2}\int_{{\tilde B}}\int_{{\tilde B}}\abs{\norm{T^\alpha_Nf(x,t)}_{\ell^\infty(\mathbb Z^2)}-\norm{T^\alpha_Nf(y,s)}_{\ell^\infty(\mathbb Z^2)}}dyds~dxdt\\
&\le {1\over |{\tilde B}|^2}\int_{\tilde B}\int_{\tilde B}{\norm{T^\alpha_Nf(x,t)-T^\alpha_Nf(y,s)}_{\ell^\infty(\mathbb Z^2)}}dyds~dxdt\\
&\le {1\over |{\tilde B}|^2}\int_{\tilde B}\int_{\tilde B}{\norm{T^\alpha_Nf_1(x,t)-T^\alpha_N f_1(y,s)}_{\ell^\infty(\mathbb Z^2)}}dyds~dxdt\\
&\quad + {1\over |{\tilde B}|^2}\int_{\tilde B}\int_{\tilde B}{\norm{T^\alpha_Nf_2(x,t)-T^\alpha_Nf_2(y,s)}_{\ell^\infty(\mathbb Z^2)}}dyds~dxdt\\
&=:I+II.
\end{align*}
The H\"older inequality and  $L^2$-boundedness of $T^*$ imply that
\begin{align*}
I&\le  {1\over |{\tilde B}|}\int_{\tilde B}{\norm{T^\alpha_Nf_1(x,t)}_{\ell^\infty(\mathbb Z^2)}}dxdt
      +{1\over |{\tilde B}|}\int_{\tilde B}{\norm{T^\alpha_Nf_1(y,s)}_{\ell^\infty(\mathbb Z^2)}}dyds\\
  &\le \left({1\over |{\tilde B}|}\int_{\tilde B}{\norm{T^\alpha_Nf_1(x,t)}^2_{\ell^\infty(\mathbb Z^2)}}dxdt\right)^{1/2}
      +\left({1\over |{\tilde B}|}\int_{\tilde B}{\norm{T^\alpha_Nf_1(y,s)}^2_{\ell^\infty(\mathbb Z^2)}}dyds\right)^{1/2}\\
  &\le C{1\over |{\tilde B}|^{1/2}}\norm{f_1}_{L^2(\mathbb R^{n+1})}\le  C\norm{f}_{BMO(\mathbb R^{n+1})}.
\end{align*}
For $II$, since $(x,t), (y,s) \in {\tilde B}$ and the support of $f_2$ is $(2{\tilde B})^c$, by a similar argument as in \eqref{equ:mean}
we have
\begin{align*}
&\Big|T_N^\alpha f_2(x,t)-T_N^\alpha f_2(y,s)\Big|\\
&=\Big|\int_{\real^{n+1}} K_N^\alpha (x-z,t-u)f_2(z,u)dzdu-\int_{\real^{n+1}} K_N^\alpha (y-z,s-u)f_2(z,u)dzdu \Big|\\
&=\Big| \int_{(2{\tilde B})^c}\left(K_N^\alpha(x-z,t-u) - K_N^\alpha (y-z,s-u)\right)f_2(z,u)dzdu\Big|\\
&\le C \int_{(2 {\tilde B})^c} \left( \frac{2r}{\abs{|z|+|u|^{1/2}}^{n+3}}+ \frac{2r^2}{\abs{|z|+|u|^{1/2}}^{n+4}}\right)\abs{f(z,u)-f_{{\tilde B}}}dzdu  \\
&\le C \sum_{k=2}^{+\infty}{ r} \int_{2^{k}  {\tilde B}\setminus 2^{k-1} {\tilde B}}  {\abs{f(z,u)-f_{ {\tilde B}}}\over ||z|+|u|^{1/2}|^{n+3}}dzdu\\
&\le C \sum_{k=2}^{+\infty}{ r\over (2^{k+1}r)^{n+3}} \int_{2^k {\tilde B}}  {\abs{f(z,u)-f_{ {\tilde B}}}}dzdu\\
&\le C \sum_{k=2}^{+\infty}2^{-(k+1)}{ 1\over \abs{2^{k} {\tilde B}}} \int_{2^{k} {\tilde B}} \left(  {\abs{f(z,u)-f_{2^k {\tilde B}}}}+\sum_{l=2}^{k}\abs{f_{2^l {\tilde B}}-f_{2^{l-1} {\tilde B}}}\right)dzdu\\
&\le C\sum_{k=2}^{+\infty}2^{-(k+1)}{ 1\over \abs{2^{k} {\tilde B}}} \int_{2^{k} {\tilde B}} \left(  {\abs{f(z,u)-f_{2^k {\tilde B}}}}+2k\norm{f}_{BMO(\real^{n+1})}\right)dzdu\\
&\le C\sum_{k=2}^{+\infty}2^{-(k+1)}{(1+2k)\norm{f}_{BMO(\real^{n+1})}}\\
&\le  C\norm{f}_{BMO(\real^{n+1})},
\end{align*}
where $2^k{\tilde B}=B(x_0, 2^kr)\times[t_0-2^kr^2, t_0+2^kr^2]$.  Hence, we have $II\le C \norm{f}_{BMO(\real^{n+1})}.$
Then by the arbitrary of $x_0$ and $r>0$, we proved
$$\norm{T^*f}_{BMO(\mathbb R^{n+1})}\le C\norm{f}_{BMO(\mathbb R^{n+1})}.$$
For the second part of $(c)$, we can deduce it from the $BMO$-boundedness of $T^*$ and the inclusion  $L^\infty(\real^{n+1})\subset BMO(\real^{n+1}).$
This completes the proof of Theorem \ref{Thm:LpBoundOsci}.
\end{proof}


\section{Proof of Theorem \ref{Thm:ae}}\label{Sec:ae}
\begin{proof}[Proof of Theorem \ref{Thm:ae}]
First, we shall see that if $\varphi$ is a test function, then $T_N^\alpha \varphi(x,t)$ converges for all $(x,t)\in \real^{n+1}$.  In order to prove this, it is enough to see that for any  $(L,M)$ with $0<L<M$,  the  series
\begin{equation*}
 A= \sum_{j=L}^M v_j ( \P^\alpha_{a_{j+1}} \varphi(x, t) - \P_{a_j}^\alpha \varphi(x,t))
 \hbox{  and  }
  B= \sum_{j=-M}^{-L} v_j ( \P^\alpha_{a_{j+1}} \varphi(x,t) - \P_{a_j}^\alpha \varphi(x,t))
 \end{equation*}
converge to zero, when $L, M\rightarrow +\infty$.
Following the arguments in the proof of Theorem \ref{Thm:Maximalcontrol}, we have
\begin{align*}
|A|   & \le  C_{n,\alpha} \norm{v}_{\ell^\infty(\mathbb Z)} \int_{\real^{n+1}}\sum_{j=L}^{M} \abs{{1\over a_{j+1}^{2+n}}-{1\over a_{j}^{2+n}}} |\varphi(x-y,t-s)|dyds \\
&\le C_{n,\alpha, v, \rho} \int_{\real^{n+1}} \sum_{j=L}^{M} \frac{1}{a^{2+n}_j}  |\varphi(x-y,t-s)|dyds\\
&\le  C_{n,\alpha, v, \rho}\left({1\over a_L^{2+n}}\sum_{j=L}^{M} \frac{a_L^{2+n}}{a_j^{2+n}}\right)\int_{\real^{n+1}}  |\varphi(x-y,t-s)|dyds\\
&\le C_{n,\alpha,v, \rho} {\rho^{2+n}\over {\rho^{2+n}-1}}\|\varphi\|_{L^1(\real^{n+1})}{1\over a_L^{2+n}} \longrightarrow 0, \quad \hbox{as}\ { L,M \to +\infty}.
\end{align*}
On the other hand,  as the integral of the kernels are zero, we can write
\begin{align*}
B&=C_{n,\alpha} \int_{\real^{n+1}} \sum_{j=-M}^{-L}v_j \frac{ a_{j+1}^{2\a}e^{-a_{j+1}^2/(4 s)}-a_j^{2\a} e^{-a_j^2/(4 s)}}{s^{1+\a}} {e^{-{|y|^n/(4s)}}\over (4\pi s)^{n/2}}( \varphi(x-y, t-s)- \varphi(x,t))dyds\\
&= C_{n,\alpha} \int_0^1  \int_{\real^n}\sum_{j=-M}^{-L}v_j \frac{ a_{j+1}^{2\a}e^{-a_{j+1}^2/(4 s)}-a_j^{2\a} e^{-a_j^2/(4 s)}}{s^{1+\a}} {e^{-{|y|^n/(4s)}}\over (4\pi s)^{n/2}}( \varphi(x-y, t-s)- \varphi(x,t))dyds\\
&\quad +C_{n,\alpha}
 \int_1^\infty\int_{\real^n}\sum_{j=-M}^{-L}v_j \frac{ a_{j+1}^{2\a}e^{-a_{j+1}^2/(4 s)}-a_j^{2\a} e^{-a_j^2/(4 s)}}{s^{1+\a}} {e^{-{|y|^n/(4s)}}\over (4\pi s)^{n/2}}( \varphi(x-y, t-s)- \varphi(x,t))dyds\\
&=: B_1+B_2.  \end{align*}
Proceeding as in the case $A$, and by using the fact that $\varphi$ is a test function,  we have
\begin{align*}
|B_1|&=  C_{n,\alpha} \Big| \int_0^1  \int_{\real^n} \sum_{j=-M}^{-L}v_j \frac{ a_{j+1}^{2\a}e^{-a_{j+1}^2/(4 s)}-a_j^{2\a} e^{-a_j^2/(4 s)}}{s^{1+\a}} {e^{-{|y|^2/(4s)}}\over (4\pi s)^{n/2}}( \varphi(x-y, t-s)- \varphi(x,t))dyds\Big| \\
&\le  C_{n,\alpha} \Big| \int_0^1  \int_{\real^n} \sum_{j=-M}^{-L}v_j \frac{ a_{j+1}^{2\a}e^{-a_{j+1}^2/(4 s)}-a_j^{2\a} e^{-a_j^2/(4 s)}}{s^{1+\a}} {e^{-{|y|^2/(4s)}}\over (4\pi s)^{n/2}}\norm{\nabla\varphi}_{L^\infty(\real^{n+1})}(|y|+s^{1/2})dyds\Big| \\
&\le C_{n,\alpha} \norm{\nabla\varphi}_{L^\infty(\real^{n+1})} \int_0^{1}\int_{\real^n} \sum_{j=-M}^{-L}v_j \frac{ a_{j+1}^{2\a}e^{-a_{j+1}^2/(4 s)}-a_j^{2\a} e^{-a_j^2/(4 s)}}{s^{\a+1/2}}{e^{-{|y|^2/(4s)}}\over (4\pi s)^{n/2}} dyds \\
&\le C_{n,\alpha,v,\rho}  \norm{\nabla\varphi}_{L^\infty(\real^{n+1})} \int_0^1 \sum_{j=-M}^{-L}  \frac{a_{j}^{2\alpha} e^{-a_{j}^2/(4 s)}}{s^{\a+1/2}}  ds.
\end{align*}
If $\displaystyle 0<\alpha\le {1\over 2},$ then, for any $0<\varepsilon<2\alpha,$ we have
\begin{align*}
|B_1|&\le C_{n,\alpha, v,\rho} \norm{\nabla\varphi}_{L^\infty(\real^{n+1})}  \sum_{j=-M}^{-L}  {a_{j}^{2\alpha-\varepsilon}}\int_0^1{s^{\varepsilon/2-\a-1/2}}  ds\\
&\le C_{n,\alpha,v,\rho,\varepsilon}\norm{\nabla\varphi}_{L^\infty(\real^{n+1})} a_{-L}^{2\alpha-\varepsilon} \sum_{j=-M}^{-L} \frac{a_{j}^{2\alpha-\varepsilon}}{a_{-L}^{2\alpha-\varepsilon}} \\
&\le C_{n,\alpha,v,\rho,\varepsilon} \norm{\nabla\varphi}_{L^\infty(\real^{n+1})}a_{-L}^{2\alpha-\varepsilon}\longrightarrow 0,\quad  \hbox{as}\ { L,M \to +\infty}.
\end{align*}
If $\displaystyle {1\over 2}<\alpha<1,$  then
\begin{align*}
|B_1|&\le C_{n,\alpha, v,\rho} \norm{\nabla\varphi}_{L^\infty(\real^{n+1})}  \int_0^1 \sum_{j=-M}^{-L}  \frac{a_{j}^{2\alpha} e^{-a_{j}^2/(4 s)}}{s^{\a+1/2}}  ds\\
&\le C_{n,\alpha,v,\rho}\norm{\nabla\varphi}_{L^\infty(\real^{n+1})} a_{-L}^{2\alpha-1} \sum_{j=-M}^{-L} \frac{a_{j}^{2\alpha-1}}{a_{-L}^{2\alpha-1}} \int_0^1 \frac{1}{s^{\a}}  ds\\
&\le C_{n,\alpha,v,\rho} \norm{\nabla\varphi}_{L^\infty(\real^{n+1})}a_{-L}^{2\alpha-1}\longrightarrow 0,\quad   \hbox{as}\ { L,M \to +\infty}.
\end{align*}
Therefore, we get
$$|B_1|\longrightarrow 0,\quad  \hbox{as}\ { L,M \to +\infty}.$$
On the other hand,
\begin{align*}
|B_2|&\le C_{n,\alpha,\rho} \norm{v}_{\ell^\infty (\mathbb Z)} \|\varphi\|_{L^\infty(\real^{n+1}) }\int_1^\infty \sum_{j=-M}^{-L}\frac{ a_{j}^{2\a}}{s^{1+\a}}  ds \\
&\le  C_{n, \alpha, v, \rho} \|\varphi\|_{L^\infty(\real^{n+1}) }\sum_{j=-M}^{-L}{ a_j^{2\alpha}}\int_1^\infty {1\over s^{1+\a}}  ds \\
& \le C_{n, \alpha, v,  \rho} \|\varphi\|_{L^\infty(\real^{n+1}) }     a_{-L}^{2\alpha}\sum_{j=-M}^{-L}  {a_j^{2\alpha}\over  a_{-L}^{2\alpha}} \\
&\le C_{n, \alpha,   v, \rho}\|\varphi\|_{L^\infty(\real^{n+1}) }   {\rho^{2\alpha}\over \rho^{2\alpha}-1}a_{-L}^{2\alpha}   \longrightarrow 0,\quad  \hbox{as}\ { L,M \to +\infty}.
\end{align*}

As the set of test functions is dense in $L^p(\mathbb{R}^{n+1})$, by Theorem \ref{Thm:LpBoundOsci} we get the $a.e.$ convergence for any function in $L^p(\mathbb{R}^{n+1})$. Analogously, since $L^p(\mathbb{R}^{n+1}) \cap L^p(\mathbb{R}^{n+1}, \omega) $ is dense in $L^p(\mathbb{R}^{n+1}, \omega)$, we get the $a.e.$ convergence for functions in $L^p(\mathbb{R}^{n+1}, \omega) $ with $1\le p<\infty$. By using the dominated convergence theorem,  we can prove the convergence in $L^p(\mathbb{R}^{n+1}, \omega)$-norm for $1<p<\infty$, and also in measure.
\end{proof}

\section{Proof of Theorem \ref{Thm:LinfinityI}}

\begin{proof}[Proof of Theorem \ref{Thm:LinfinityI}]
 Let
 $$g(t)=\sum_{k\in \mathbb{Z}} (-1)^{k} \chi_{(-a^{2k+1}, -a^{2k}]}(t),\quad t\in \real,$$
 and
\begin{equation*}
f(x,t)=\chi_{\real^n}(x)g(t),\quad (x,t)\in \real^{n+1},
\end{equation*}
where $a>1$ is a real number that we shall fix it later. Then $f\in L^\infty(\real^{n+1}).$
It is easy to see that
\begin{equation*}
f(a^jx, a^{2j}t) = (-1)^j f(x,t).
\end{equation*}
Let $a_j= a^{j}.$ Then
\begin{align}\label{equ:dilation}
\mathcal{P}_{a_j}^\alpha f(x,t) &= \frac{1}{4^\alpha \Gamma(\alpha) }
\int_0^{\infty}\int_{\real^n} \frac{a^{2\alpha j} e^{-a^{2j}/(4s)}}{s^{1+\alpha}} {e^{-{|y|^2/(4s)}}\over (4\pi s)^{n/2}}f(x-y,t-s) dyds\nonumber\\
& =\frac{1}{4^\alpha \Gamma(\alpha) }\int_0^{\infty}\int_{\real^n} \frac{ e^{-1/(4u)}}{u^{\alpha}}{e^{-{|z|^2/(4u)}}\over (4\pi u)^{n/2}} f(x-a^jz,t-a^{2j}u)
     dz\frac{du}{u}\nonumber\\
     & =\frac{(-1)^j}{4^\alpha \Gamma(\alpha) }\int_0^{\infty}\int_{\real^n} \frac{ e^{-1/(4u)}}{u^{\alpha}}{e^{-{|z|^2/(4u)}}\over (4\pi u)^{n/2}} f\left({x\over a^j}-z,{t\over a^{2j}}-u\right)
     dz\frac{du}{u}\\
&=\frac{(-1)^j}{4^\alpha \Gamma(\alpha) }\int_0^{\infty}\int_{\real^n} \frac{ e^{-1/(4u)}}{u^{\alpha}}{e^{-{\abs{{x\over a^j}-z}^2/(4u)}}\over (4\pi u)^{n/2}} g\left({t\over a^{2j}}-u\right)dz\frac{du}{u}\nonumber\\
&=\frac{2(-1)^j}{4^\alpha \Gamma(\alpha) }\int_0^{\infty} \frac{ e^{-1/(4u)}}{u^{\alpha}} g\left({t\over a^{2j}}-u\right) \frac{du}{u}.\nonumber
\end{align}
Therefore,
\begin{align*}
\mathcal{P}_{a_j}^\alpha f(0,0)
&=\frac{2(-1)^j}{4^\alpha \Gamma(\alpha) }\int_0^{\infty} \frac{ e^{-1/(4u)}}{u^{\alpha}} g(-u)
     \frac{du}{u}.
\end{align*}
We observe that
\begin{eqnarray*}
\int_0^{\infty} \frac{ e^{-1/(4u)}}{u^{\alpha}} \big| g(-u)\big| \frac{du}{u} \le
\int_0^{\infty} \frac{ e^{-1/(4u)}}{u^{\alpha}} \frac{du}{u}=4^\alpha\Gamma(\alpha)< \infty.
\end{eqnarray*}
Hence $$\lim_{R\to {+\infty}} \int_R^{\infty} \frac{ e^{-1/(4u)}}{u^{\alpha}}  g(-u)\frac{du}{u} =0\quad  \hbox{  and  } \quad \lim_{\varepsilon \to 0^+} \int_0^\varepsilon \frac{ e^{-1/(4u)}}{u^{\alpha}}  g(-u)\frac{du}{u} =0. $$
On the other hand, $\displaystyle \lim_{a\to {+\infty}} \int_1^a \frac{ e^{-1/(4u)}}{u^{\alpha}}  f(-u)\frac{du}{u}  = \displaystyle \lim_{a\to {+\infty}} \int_1^a \frac{ e^{-1/(4u)}}{u^{\alpha}} \frac{du}{u} =C>0.$ Hence we can choose $a>1$ big enough such that
 \begin{multline*}
 \int_1^a \frac{ e^{-1/(4u)}}{u^{\alpha}}  g(-u)\frac{du}{u} = \int_1^a \frac{ e^{-1/(4u)}}{u^{\alpha}}  \frac{du}{u}  >  \abs{\int_0^{1/a} \frac{ e^{-1/(4u)}}{u^{\alpha}} \frac{du}{u}} + \abs{\int_{a^2}^{+\infty} \frac{ e^{-1/(4u)}}{u^{\alpha}} \frac{du}{u}}\\
 >  \abs{\int_0^{1/a} \frac{ e^{-1/(4u)}}{u^{\alpha}}g(-u) \frac{du}{u}} + \abs{\int_{a^2}^{+\infty} \frac{ e^{-1/(4u)}}{u^{\alpha}}g(-u) \frac{du}{u}}.
\end{multline*}
In other words, with the $a>1$ fixed above,  there exists constant $C_1>0$ such that
\begin{equation}\label{equ:cons1}
\int_0^{+\infty} \frac{ e^{-1/(4u)}}{u^{\alpha}}  g(-u)\frac{du}{u}=C_1.
\end{equation}
Hence
\begin{equation*}
\Big|\mathcal{P}_{a_j}^\alpha f(0,0)- \mathcal{P}_{a_{j+1}}^\alpha f(0,0)\Big| = \frac{ 4 C_1}{4^\alpha \Gamma(\alpha)}>0.
\end{equation*}
Therefore we have   $$\sum_{j\in \mathbb Z} \Big|\mathcal{P}_{a_{j+1}}^\alpha f(0,0)- \mathcal{P}_{a_{j}}^\alpha f(0,0)\Big|  = \infty.$$
By  \eqref{equ:dilation}, we get
\begin{align}\label{equ:diffP}
&{\mathcal{P}_{a_{j+1}}^\alpha f(x,t)-\mathcal{P}_{a_j}^\alpha f(x,t)}\nonumber\\
& =
\frac{ (-1)^{j+1}}{4^\alpha \Gamma(\alpha) }
\Big\{\int_0^{+\infty}\frac{ e^{-1/(4u)}}{u^{\alpha}} g\left({t\over {a^{2(j+1)}}}-u\right) \frac{du}{u}+\int_0^{+\infty}\frac{ e^{-1/(4u)}}{u^{\alpha}} g\left({t\over {a^{2j}}}-u\right) \frac{du}{u}\Big\}.
\end{align}
By the dominated convergence theorem, we know that
\begin{equation*}
\lim_{h\to 0}\int_0^{+\infty}\frac{ e^{-1/(4u)}}{u^{\alpha}} g(h-u) \frac{du}{u}=\int_0^{+\infty}\frac{ e^{-1/(4u)}}{u^{\alpha}} g(-u) \frac{du}{u}  = C_1>0,
\end{equation*}
where $C_1$ is the constant appeared in \eqref{equ:cons1}.
So, there exists $0<\eta_0<1,$ such that, for $|h|<\eta_0,$
\begin{equation*}
\int_0^{+\infty}\frac{ e^{-1/(4u)}}{u^{\alpha}} g(h-u) \frac{du}{u}\ge {1\over 2}\int_0^{+\infty}\frac{ e^{-1/(4u)}}{u^{\alpha}} g(-u) \frac{du}{u} = \frac{C_1}{2}.
\end{equation*}
Then, for each $t\in \real$, we can choose $j\in \mathbb Z$  such that $\displaystyle {|t|\over {a^j}}<\eta_0$  (there are infinite $j$ satisfying this condition), and we have
\begin{equation*}
\int_0^{+\infty}\frac{ e^{-1/(4u)}}{u^{\alpha}} g\left({t\over {a^{2(j+1)}}}-u\right) \frac{du}{u}+\int_0^{+\infty}\frac{ e^{-1/(4u)}}{u^{\alpha}} g\left({t\over {a^{2j}}}-u\right) \frac{du}{u} \ge C_1>0.
\end{equation*}
Choosing $v_j=(-1)^{j+1},\ j\in \mathbb Z$, by \eqref{equ:diffP} we have, for any $t\in \mathbb R,$
\begin{align*}
T^* f(x,t)&\ge \sum_{\abs{t\over {a^j}}<\eta_0} (-1)^{j+1}\big(\P_{a_{j+1}}^\alpha f(x,t)-\P_{a_j}^\alpha f(x,t)\big)\\
&=\frac{ 2}{4^\alpha \Gamma(\alpha) }\sum_{\abs{t\over {a^j}}<\eta_0}
\left({\int_0^{+\infty}\frac{ e^{-1/(4u)}}{u^{\alpha}} g\left({t\over {a^{2(j+1)}}}-u\right) \frac{du}{u}+\int_0^{+\infty}\frac{ e^{-1/(4u)}}{u^{\alpha}} g\left({t\over {a^{2j}}}-u\right) \frac{du}{u}}\right)\\
&=\infty.
\end{align*}
We complete the proof of Theorem \ref{Thm:LinfinityI}.
\end{proof}


\section{Local growth of the operator $T^*$ }\label{Sec:L00}


\begin{proof}[Proof of Theorem \ref{Thm:GrothLinfinity}.]
 For $(a)$, we will prove it only in the case $1<p<\infty.$  Since $2r<1,$ we know that $\tilde B\backslash \tilde B_{2r}\neq \emptyset.$ Let $f(x,t)=f_1(x,t)+f_2(x,t)$, where $f_1(x,t)=f(x,t)\chi_{\tilde B_{2r}}(x,t)$ and $f_2(x,t)=f(x,t)\chi_{\tilde B\backslash \tilde  B_{2r}}(x,t)$. Then
$$\abs{T^* f(x,t)}\le \abs{T^* f_1(x,t)}+\abs{T^* f_2(x,t)}.$$
By Theorem \ref{Thm:LpBoundOsci}, we have
\begin{multline*}
\frac{1}{|\tilde B_r|} \int_{\tilde B_r} \abs{T^* f_1 (x,t)} dx\le \left(\frac{1}{|\tilde B_r|} \int_{\tilde B_r} \abs{T^* f_1 (x,t)}^2 dxdt\right)^{1/2}\\
\le C \left(\frac{1}{|\tilde B_r|} \int_{\mathbb{R}^{n+1}}\abs{ f_1 (x,t)}^2 dxdt\right)^{1/2}\le C\norm{f}_{L^\infty(\mathbb{R}^{n+1})}.
\end{multline*}
Also, we have
\begin{multline}\label{equ:cosntant}
\int_{\real^{n+1}_+}\abs{\frac{ a_{j+1}^{2\a}e^{-a_{j+1}^2/(4 s)}-a_j^{2\a} e^{-a_j^2/(4 s)}}{s^{1+\a}}} {e^{-|y|^2/(4s)}\over (4\pi s)^{n/2}}dyds\\
\le \int_{\real^{n+1}_+}{\frac{ a_{j+1}^{2\a}e^{-a_{j+1}^2/(4 s)}+a_j^{2\a} e^{-a_j^2/(4 s)}}{s^{1+\a}}} {e^{-|y|^2/(4s)}\over (4\pi s)^{n/2}}dyds= 4^{1+\alpha}\Gamma(\alpha).
\end{multline}
Then, by H\"older's inequality, \eqref{equ:cosntant}  and Fubini's Theorem, for $1< p < \infty$ and any $N=(N_1, N_2)$, we have
\begin{align*}\label{equ:Palpha}
&\abs{\sum_{j=N_1}^{N_2}v_j\left(\P_{a_{j+1}}^\a f_2(x,t)-\P_{a_j}^\a f_2(x,t)\right)}\nonumber\\
&\le C\sum_{j=N_1}^{N_2} \abs{v_j \int_{\real^{n+1}_+}\frac{ a_{j+1}^{2\a}e^{-a_{j+1}^2/(4 s)}-a_j^{2\a} e^{-a_j^2/(4 s))}}{s^{1+\a}}{e^{-|y|^2/(4s)}\over (4\pi s)^{n/2}} f_2(x-y, t-s)~ dyds}\nonumber\\
&\le C\norm{v}_{l^p(\mathbb Z)}\left( \sum_{j=N_1}^{N_2}\left(\int_{\real^{n+1}_+}\abs{\frac{ a_{j+1}^{2\a}e^{-a_{j+1}^2/(4 s)}-a_j^{2\a} e^{-a_j^2/(4 s)}}{s^{1+\a}}  } {e^{-|y|^2/(4s)}\over (4\pi s)^{n/2}}\abs{f_2(x-y,t-s)}~ dyds\right)^{p'}\right)^{1/p'} \nonumber \\
&\le C\norm{v}_{l^p(\mathbb Z)}\Big(\sum_{j=N_1}^{N_2} \Big\{\int_{\real^{n+1}_+}\abs{\frac{ a_{j+1}^{2\a}e^{-a_{j+1}^2/(4 s)}-a_j^{2\a} e^{-a_j^2/(4 s)}}{s^{1+\a}}  } {e^{-|y|^2/(4s)}\over (4\pi s)^{n/2}}\abs{f_2(x-y, t-s)}^{p'}~ dyds\Big\}  \\
 & \quad \quad  \times  \Big\{\int_{\real^{n+1}_+}\abs{\frac{ a_{j+1}^{2\a}e^{-a_{j+1}^2/(4 s)}-a_j^{2\a} e^{-a_j^2/(4 s)}}{s^{1+\a}}} {e^{-|y|^2/(4s)}\over (4\pi s)^{n/2}}dyds  \Big\}^{p'/p} \Big)^{1/p'}
\nonumber \\
&\le C \norm{v}_{l^p(\mathbb Z)}\left(\sum_{j=N_1}^{N_2} \int_{\real^{n+1}_+}\abs{\frac{ a_{j+1}^{2\a}e^{-a_{j+1}^2/(4 s)}-a_j^{2\a} e^{-a_j^2/(4 s)}}{s^{1+\a}}  }{e^{-|y|^2/(4s)}\over (4\pi s)^{n/2}} \abs{f_2(x-y, t-s)}^{p'}~ dyds   \right)^{1/p'}
\nonumber \\
&\le C \norm{v}_{l^p(\mathbb Z)}\left(\int_{\real^{n+1}_+}\sum_{j=-\infty}^{+\infty} \abs{\frac{ a_{j+1}^{2\a}e^{-a_{j+1}^2/(4 s)}-a_j^{2\a} e^{-a_j^2/(4 s)}}{s^{1+\a}}  } {e^{-|y|^2/(4s)}\over (4\pi s)^{n/2}} \abs{f_2(x-y, t-s)}^{p'}~ dyds  \right)^{1/p'}
\nonumber \\ \nonumber
&\le C \norm{v}_{l^p(\mathbb Z)}\left(\int_{\real^{n+1}_+} \frac1{|s|} {e^{-|y|^2/(4s)}\over (4\pi s)^{n/2}} \abs{f_2(x-y, t-s)}^{p'}~ dyds  \right)^{1/p'}.
\end{align*}
 Then,  we get
\begin{align*}
&\frac{1}{|\tilde B_r|} \int_{\tilde B_r} \abs{T^* f_2(x,t)}dxdt \\
&\le  C\frac{1}{|\tilde B_r|} \int_{\tilde B_r} \left(\int_{\real^{n+1}_+} \frac1{|s|} {e^{-|y|^2/(4s)}\over (4\pi s)^{n/2}} \abs{f_2(x-y, t-s)}^{p'}~ dyds \right)^{1/p'}dxdt \\
&=  C\frac{1}{|\tilde B_r|} \int_{\tilde B_r} \left(\int_{\real^{n+1}_+} \frac1{|t-s|} {e^{-|x-y|^2/(4(t-s))}\over (4\pi (t-s))^{n/2}}\abs{f_2(y,s)}^{p'}~ dyds  \right)^{1/p'}dxdt \\
&\le C\frac{\norm{f}_{L^\infty(\real^{n+1})}}{|\tilde B_r|} \int_{\tilde B_r} \left(\int_{r\le |t-s| \le 2} \frac1{|t-s|} ~ ds  \right)^{1/p'}dxdt \\
 &\sim \Big(\log\frac{2}{r}\Big)^{1/p'}\norm{f}_{L^\infty(\real^{n+1})}.
\end{align*}
In the above inequalities, we have used that, if $(y,s)\in \tilde B\backslash \tilde B_{2r}$ and $(x,t)\in \tilde B_r$,  then $r\le |t-s|\le 2$.
Hence, $$\frac{1}{|\tilde B_r|} \int_{\tilde B_r} \abs{T^* f(x,t)}dxdt\le C\left(1+\Big(\log\frac{2}{r}\Big)^{1/p'}\right)\norm{f}_{L^\infty(\real^{n+1})}  \le C\Big(\log\frac{2}{r}\Big)^{1/p'}\norm{f}_{L^\infty(\real^{n+1})}.$$
For the case $p=1$ and $p=\infty$, the proof is similar and easier. Then we get the proof of $(a).$

For $(b),$ we will only consider the two dimensional case. When $1< p<\infty,$ for any $0<\varepsilon<p-1$, let
\begin{equation*}
f(x,t) = \sum_{k=-\infty}^{0} (-1)^{k} \chi_{(-a^k,-a^{k-1}]\times(-a^{2k}, -a^{2k-1}]}(x,t)\quad \hbox{and }\quad a_j=a^{j}
\end{equation*}
with $a>1$ being fixed later. Then, the support of $f$ is contained in $(-1, 0)\times (-1, 0),$ and $\{a_j\}_{j\in \mathbb Z}$ is a $\rho$-lacunary sequence with $\rho=a>1.$ We observe that
\begin{multline*}
\abs{\int_0^{+\infty}\int_\real \frac{ e^{-1/(4u)}}{u^{\alpha}}  \frac{e^{-{z^2\over 4u}}}{u^{1/2}} f(-z,-u) dz\frac{du}{u}}=\abs{\int_0^{+\infty}\int_0^{+\infty} \frac{ e^{-1/(4u)}}{u^{\alpha}}  \frac{e^{-{z^2\over 4u}}}{u^{1/2}} f(-z,-u) dz\frac{du}{u}}\\
\le\int_0^{+\infty}\int_0^{+\infty} \frac{ e^{-1/(4u)}}{u^{\alpha}} \frac{e^{-{z^2\over 4u}}}{u^{1/2}}dz\frac{du}{u}=\sqrt \pi 4^\alpha\Gamma(\alpha)< \infty.
\end{multline*}
Also, we have  $$\lim_{R\to {+\infty}} \int_R^{+\infty} \int_0^{+\infty}\frac{ e^{-1/(4u)}}{u^{\alpha}}  \frac{e^{-{z^2\over 4u}}}{u^{1/2}}dz\frac{du}{u} =\sqrt \pi \lim_{R\to {+\infty}} \int_R^{+\infty} \frac{ e^{-1/(4u)}}{u^{\alpha}}  \frac{du}{u}=0$$
and
 $$\lim_{\varepsilon\to 0^+} \int_0^\varepsilon \int_0^{+\infty} \frac{ e^{-1/(4u)}}{u^{\alpha}} \frac{e^{-{z^2\over 4u}}}{u^{1/2}}dz\frac{du}{u} =0. $$
 Hence, we can choose $a>1$ big enough such that
 \begin{align*}
&\int_{a^{-1}}^1\int_{a^{-1}}^1 \frac{ e^{-1/(4u)}}{u^{\alpha}}  \frac{e^{-{z^2\over 4u}}}{u^{1/2}} f(-z,-u) dz\frac{du}{u} = \int_{a^{-1}}^1\int_{a^{-1}}^1\frac{ e^{-1/(4u)}}{u^{\alpha}} \frac{e^{-{z^2\over 4u}}}{u^{1/2}}  dz \frac{du}{u}\\
& \ge 5 \left(\int_0^{1/a^2}\int_0^{+\infty} \frac{ e^{-1/(4u)}}{u^{\alpha}}\frac{e^{-{z^2\over 4u}}}{u^{1/2}}dz \frac{du}{u}+\int_{a-1}^{+\infty}\int_{0}^{+\infty} \frac{ e^{-1/(4u)}}{u^{\alpha}} \frac{e^{-{z^2\over 4u}}}{u^{1/2}}dz\frac{du}{u}\right)\nonumber\\
& >  5 \left( \abs{\int_0^{1/a^2} \int_0^{+\infty} \frac{ e^{-1/(4u)}}{u^{\alpha}}\frac{e^{-{z^2\over 4u}}}{u^{1/2}}f(-z,-u)dz \frac{du}{u}}+\abs{\int_{a-1}^{+\infty}\int_{0}^{+\infty} \frac{ e^{-1/(4u)}}{u^{\alpha}} \frac{e^{-{z^2\over 4u}}}{u^{1/2}}f(-z,-u)dz\frac{du}{u}}\right).
\end{align*}
We should note that
\begin{align*}
&\int_0^{+\infty}\int_0^{+\infty} \frac{ e^{-1/(4u)}}{u^{\alpha}} \frac{e^{-{z^2\over 4u}}}{u^{1/2}}f(-z,-u)dz\frac{du}{u}\\
&= \int_{a^{-1}}^1\int_{a^{-1}}^1\frac{ e^{-1/(4u)}}{u^{\alpha}} \frac{e^{-{z^2\over 4u}}}{u^{1/2}} f(-z,-u) dz \frac{du}{u}+{\int_0^{1/a^2}\int_0^{a^{-1}} \frac{ e^{-1/(4u)}}{u^{\alpha}}\frac{e^{-{z^2\over 4u}}}{u^{1/2}}f(-z,-u)dz \frac{du}{u}}\\
&\quad +{\int_{a}^{+\infty}\int_{1}^{+\infty} \frac{ e^{-1/(4u)}}{u^{\alpha}} \frac{e^{-{z^2\over 4u}}}{u^{1/2}}f(-z,-u)dz\frac{du}{u}}\\
&\ge \int_{a^{-1}}^1\int_{a^{-1}}^1\frac{ e^{-1/(4u)}}{u^{\alpha}} \frac{e^{-{z^2\over 4u}}}{u^{1/2}}  dz \frac{du}{u}-\int_0^{1/a^2}\int_0^{a^{-1}} \frac{ e^{-1/(4u)}}{u^{\alpha}}\frac{e^{-{z^2\over 4u}}}{u^{1/2}}dz \frac{du}{u}\\
&\quad -\int_{a}^{+\infty}\int_{1}^{+\infty} \frac{ e^{-1/(4u)}}{u^{\alpha}} \frac{e^{-{z^2\over 4u}}}{u^{1/2}}dz\frac{du}{u}\\
 &\ge \int_{a^{-1}}^1\int_{a^{-1}}^1\frac{ e^{-1/(4u)}}{u^{\alpha}} \frac{e^{-{z^2\over 4u}}}{u^{1/2}}  dz \frac{du}{u}-\int_0^{1/a^2}\int_0^{+\infty} \frac{ e^{-1/(4u)}}{u^{\alpha}}\frac{e^{-{z^2\over 4u}}}{u^{1/2}}dz \frac{du}{u}\\
&\quad -\int_{a-1}^{+\infty}\int_{0}^{+\infty} \frac{ e^{-1/(4u)}}{u^{\alpha}} \frac{e^{-{z^2\over 4u}}}{u^{1/2}}dz\frac{du}{u}>0.
\end{align*}
Therefore, there exists a constant $C_1>0$ such that
\begin{equation}\label{equ:cons2}
\int_0^{+\infty}\int_0^{+\infty} \frac{ e^{-1/(4u)}}{u^{\alpha}} \frac{e^{-{z^2\over 4u}}}{u^{1/2}}f(-z,-u)dz\frac{du}{u}=C_1>0
\end{equation}
and
\begin{equation}\label{equ:cons33}
0<\int_0^{1/a^2}\int_0^{+\infty} \frac{ e^{-1/(4u)}}{u^{\alpha}}\frac{e^{-{z^2\over 4u}}}{u^{1/2}}dz \frac{du}{u}+\int_{a-1}^{+\infty}\int_{0}^{+\infty} \frac{ e^{-1/(4u)}}{u^{\alpha}} \frac{e^{-{z^2\over 4u}}}{u^{1/2}}dz\frac{du}{u}\le {C_1\over 4}.
\end{equation}
On the other hand,
by the dominated convergence theorem, we have
\begin{multline*}
\lim_{h\to 0}\int_0^{+\infty}\int_0^{+\infty}\frac{ e^{-1/(4u)}}{u^{\alpha}} \frac{e^{-{z^2\over 4u}}}{u^{1/2}}f(h-z,h-u)dz\frac{du}{u}\\
=\int_0^{+\infty}\int_0^{+\infty}\frac{ e^{-1/(4u)}}{u^{\alpha}} \frac{e^{-{z^2\over 4u}}}{u^{1/2}}f(-z,-u)dz\frac{du}{u} = C_1>0,
\end{multline*}
where $C_1$ is the constant appeared in \eqref{equ:cons2}.
So, there exists $0<\eta_0<1,$ such that, for $|h|<\eta_0,$
\begin{equation}\label{eq:bigC}
\int_0^{+\infty}\int_0^{+\infty}\frac{ e^{-1/(4u)}}{u^{\alpha}} \frac{e^{-{z^2\over 4u}}}{u^{1/2}}f(h-z,h-u)dz\frac{du}{u}\ge \frac{C_1}{2}.
\end{equation}

It can be checked that
$$f(a^jx, a^{2j}t) = (-1)^j f(x,t)+(-1)^j\sum_{k=1}^{-j}(-1)^k \chi_{(-a^k, -a^{k-1}]\times (-a^{2k}, -a^{2k-1}]}(x,t),$$
when $j\le 0.$ We will always assume $j\le 0$ in the following.
By changing variable,
\begin{align*}
\mathcal{P}_{a_j}^\alpha f(x,t) &=
C_\alpha \int_0^{+\infty}\int_\real\frac{ e^{-1/(4u)}}{u^{\alpha}} \frac{e^{-{z^2\over 4u}}}{u^{1/2}} f(x-a^jz,t-a^{2j}u)dz \frac{du}{u}\\
& ={(-1)^j }C_\alpha
\int_0^{+\infty}\int_\real \frac{ e^{-1/(4u)}}{u^{\alpha}}\frac{e^{-{z^2\over 4u}}}{u^{1/2}} \Big\{ f\left({x\over a^j}-z,{t\over a^{2j}}-u\right)\\
&\quad +\sum_{k=1}^{-j}(-1)^k \chi_{(-a^k, -a^{k-1}]\times (-a^{2k}, -a^{2k-1}]}\left({x\over a^j}-z,{t\over a^{2j}}-u\right) \Big\}dz\frac{du}{u}.
\end{align*}
Then
\begin{align}\label{equ:integ}
&{\mathcal{P}_{a_{j+1}}^\alpha f(x,t)-\mathcal{P}_{a_j}^\alpha f(x,t)}\nonumber\\
& =
{ (-1)^{j+1}}C_\alpha
\Big\{\, \int_0^{+\infty}\int_\real \frac{ e^{-1/(4u)}}{u^{\alpha}}\frac{e^{-{z^2\over 4u}}}{u^{1/2}}  f\left({x\over a^{j+1}}-z,{t\over {a^{2(j+1)}}}-u\right)  dz\frac{du}{u}\nonumber \\
&\quad +\int_0^{+\infty}\int_\real \frac{ e^{-1/(4u)}}{u^{\alpha}}\frac{e^{-{z^2\over 4u}}}{u^{1/2}}  f\left({x\over a^{j}}-z,{t\over {a^{2j}}}-u\right)dz\frac{du}{u} \\
&\quad +\int_0^{+\infty}\int_\real \frac{ e^{-1/(4u)}}{u^{\alpha}} \frac{e^{-{z^2\over 4u}}}{u^{1/2}} \sum_{k=1}^{-j-1}(-1)^k \chi_{(-a^k, -a^{k-1}]\times(-a^{2k}, -a^{2k-1}]}\left({x\over a^{j+1}}-z,{t\over a^{2j+2}}-u\right)\frac{du}{u}\nonumber\\
&\quad +\int_0^{+\infty}\int_\real\frac{ e^{-1/(4u)}}{u^{\alpha}}\frac{e^{-{z^2\over 4u}}}{u^{1/2}} \sum_{k=1}^{-j}(-1)^k \chi_{(-a^k, -a^{k-1}]\times(-a^{2k}, -a^{2k-1}]}\left({x\over a^{j}}-z, {t\over a^{2j}}-u\right) \frac{du}{u}\, \Big\}.\nonumber
\end{align}
For given $\eta_0$ as above, let $2r<1$ such that  $r< \eta_0^2$   and $r \sim  a^{2J_0}\eta_0$ for a certain negative integer $J_0$.  If $J_0\le j\le 0$, we have   $\displaystyle {r\over a^{2j}} <\eta_0$.
And,  for any $-r\le x, t\le r$  we have
$$\displaystyle -1\cdot \chi_{(0, +\infty)\times [a-1,+\infty)}(z, u)\le \sum_{k=1}^{-j-1}(-1)^k\chi_{(-a^k, -a^{k-1}]\times(-a^{2k}, -a^{2k-1}]}\left({x\over a^{j+1}}-z,{t\over a^{2j+2}}-u\right)$$
 and
  $$\displaystyle -1\cdot \chi_{(0, +\infty)\times [a-1,+\infty)}(u) \le \sum_{k=1}^{-j} (-1)^k\chi_{(-a^k, -a^{k-1}]\times(-a^{2k}, -a^{2k-1}]}\left({x\over a^{j}}-z,{t\over a^{2j}}-u\right). $$
Hence, for the third and fourth integrals in \eqref{equ:integ}, by \eqref{equ:cons33} we have
\begin{align} \label{equ:large}
&\int_0^{+\infty}\int_\real \frac{ e^{-1/(4u)}}{u^{\alpha}} \frac{e^{-{z^2\over 4u}}}{u^{1/2}} \sum_{k=1}^{-j-1}(-1)^k \chi_{(-a^k, -a^{k-1}]\times(-a^{2k}, -a^{2k-1}]}\left({x\over a^{j+1}}-z,{t\over a^{2j+2}}-u\right)\frac{du}{u}+\nonumber\\
& \int_0^{+\infty}\int_\real\frac{ e^{-1/(4u)}}{u^{\alpha}}\frac{e^{-{z^2\over 4u}}}{u^{1/2}} \sum_{k=1}^{-j}(-1)^k \chi_{(-a^k, -a^{k-1}]\times(-a^{2k}, -a^{2k-1}]}\left({x\over a^{j}}-z, {t\over a^{2j}}-u\right) \frac{du}{u}\nonumber\\
&\quad \quad \quad \quad\quad \quad\quad \quad\quad \quad \quad \quad\quad \quad \quad \quad \quad \quad\ge (-2)\int_{a-1}^{+\infty}\int_{0}^{+\infty}\frac{ e^{-1/(4u)}}{u^{\alpha}}\frac{e^{-{z^2\over 4u}}}{u^{1/2}}dz\frac{du}{u}\ge -{C_1\over 2}.
\end{align}
So, for any $(x,t)\in [-r, r]\times [-r, r]$ and $J_0\le j\le 0$, combining \eqref{equ:integ},  \eqref{eq:bigC} and \eqref{equ:large},  we have
\begin{align*}
\abs{\mathcal{P}_{a_{j+1}}^\alpha f(x,t)-\mathcal{P}_{a_j}^\alpha f(x,t)} \ge C_\alpha \cdot\left( C_1-{C_1\over 2}\right)=C\cdot C_1>0.
\end{align*}
We choose the sequence  $\{v_j\}_{j\in \mathbb Z} \in \ell^p(\mathbb Z)$ given by  $\displaystyle v_j=(-1)^{j+1}(-j)^{-{1\over p-\varepsilon}}$, then  for  $N=(J_0, 0),$   we have
\begin{multline*}
 \frac1{4r^2}\int_{[-r,r]}\int_{[-r,r]} \abs{T^* f(x,t)} dx dt \ge \frac1{4r^2}\int_{[-r,r]} \abs{T_N^\alpha f(x, t)} dxdt \\
 \ge  C_\alpha \frac 1{4r^2}\int_{[-r,r]}\int_{[-r,r]}  \sum_{j =J_0}^{-1} \left(C\cdot C_1\cdot (-j)^{-{1\over p-\varepsilon}}\right) dxdt\\ \ge C_{p,\varepsilon,\alpha}\cdot C_1\cdot (-J_0)^{1\over {(p-\varepsilon)'}}  \sim \left(\log \frac 2{r}\right)^{1\over {(p-\varepsilon)'}}.
 \end{multline*}

For $(c)$,  let $v_j=(-1)^{j+1}$,  $a_j=a^{j}$ with $a>1$  and  $0<\eta_0<1$ fixed in the proof of $(b)$.
Consider the same function $f$ as in $(b).$  Then, $\norm{v}_{\ell^\infty(\mathbb Z)}=1$ and $\norm{f}_{L^\infty(\mathbb R^2)}=1.$
By the same argument as in $(b)$, with $N=(J_0, 0)$ and $0<\alpha<1$,  we have
\begin{multline*}
\frac1{4r^2}\int_{[-r,r]}\int_{[-r,r]} \abs{T^* f(x,t)} dxdt\ge \frac1{4r^2}\int_{[-r,r]}\int_{[-r,r]} \abs{T_N^\alpha f(x,t)} dxdt \\
\ge  C_\alpha \frac 1{4r^2}\int_{[-r,r]} \int_{[-r,r]} \sum_{j = J_0}^{0} C_1 dxdt \ge  C_\alpha C_1 \cdot (-J_0)  \sim \log \frac 2{r}.
\end{multline*}
\end{proof}

\vspace{1em}

\noindent{\bf Acknowledgments.}   The author is grateful to Professors J. L. Torrea and T. Ma for their  helpful discussions.

\vspace{3em}


\end{document}